\documentclass{aic}

\aicAUTHORdetails{%
  title = {Decomposing a Signed Graph into Rooted Circuits}, %% please capitalize all significant words
  author = {Rose McCarty},
  plaintextauthor = {Rose McCarty},
  plaintexttitle = {Decomposing a Signed Graph into Rooted Circuits}, %%  title without math or LaTeX
  keywords = {Eulerian graphs, signed graphs, decompositions, vertex-minors},
}   %%% END \aicAUTHORdetails

\aicEDITORdetails{%
   year={2024},
   number={7},
   received={13 August 2023},   % received date: example: 7 January 2017
   published={20 December 2024},  % published date
   updated={6 June 2025}
}   %%% END \aicEDITORdetails

\usepackage{mathtools}
\usepackage{thmtools}
\usepackage{thm-restate}
\newcommand{\rooted}{RES-graph{}}
\newcommand{\rooteds}{RES-graphs{}}
\newcommand{\C}{\mathcal{C}}

\newtheorem{theorem}{Theorem}[section]
\newtheorem{lemma}[theorem]{Lemma}

\newtheorem{corollary}[theorem]{Corollary}

\newtheorem{problem}[theorem]{Problem}
\newtheorem{claim}[theorem]{Claim}

\usepackage{tikz, tkz-euclide, adjustbox}
\usetikzlibrary{hobby}
\usetikzlibrary{decorations.markings}
\usetikzlibrary{arrows.meta}
\tikzset{MyNode/.style={circle, draw, inner sep=2,outer sep=0, fill=gray}}
\tikzset{MyRedArc/.style={line width=2.5pt, red}}
\tikzset{MyBlueArc/.style={line width=2pt, blue}}
\tikzset{
  arrBlack/.style={
    postaction={
      decorate,
      decoration={
        markings,
        mark= at position .18 with {\arrow{>}},
        mark= at position .48 with {\arrow{>}},
        mark= at position .81 with {\arrow{>}},
}}}}
\tikzset{
  arrRed/.style={
    postaction={
      decorate,
      decoration={
        markings,
        mark= at position .26 with {\arrow[scale=.75]{>}},
        mark= at position .75 with {\arrow[scale=.75]{>}},
}}}}
\tikzset{
  arrRepeat/.style={
    postaction={
      decorate,
      decoration={
        markings,
        mark=between positions 0.26 and 0.85 step 31pt with {\arrow{>}},
}}}}
\tikzset{
  arrRepeatReverse/.style={
    postaction={
      decorate,
      decoration={
        markings,
        mark=between positions 0.26 and 0.85 step 31pt with {\arrow{<}},
}}}}
\tikzset{
  arrRepeatIn/.style={
    postaction={
      decorate,
      decoration={
        markings,
        mark=between positions 0.42 and 0.85 step 28pt with {\arrow{>}},
}}}}
\tikzset{
  arrRepeatInReverse/.style={
    postaction={
      decorate,
      decoration={
        markings,
        mark=between positions 0.42 and 0.85 step 28pt with {\arrow{<}},
}}}}
\tikzset{
  arrLoop/.style={
    postaction={
      decorate,
      decoration={
        markings,
        mark=between positions 0.18 and 0.9 step 30pt with {\arrow{>}},
}}}}
\tikzset{
  arrLoopReverse/.style={
    postaction={
      decorate,
      decoration={
        markings,
        mark=between positions 0.18 and 0.9 step 30pt with {\arrow{<}},
}}}}
\tikzset{
  arrLoopBot/.style={
    postaction={
      decorate,
      decoration={
        markings,
        mark=between positions 0.3 and 0.9 step 30pt with {\arrow{>}},
}}}}
\tikzset{
  arrLoopBotReverse/.style={
    postaction={
      decorate,
      decoration={
        markings,
        mark=between positions 0.3 and 0.9 step 30pt with {\arrow{<}},
}}}}
\tikzset{
  arrLoopSmaller/.style={
    postaction={
      decorate,
      decoration={
        markings,
        mark=between positions 0.26 and 0.9 step 23pt with {\arrow{>}},
}}}}
\tikzset{
  arrLoopMid/.style={
    postaction={
      decorate,
      decoration={
        markings,
        mark=between positions 0.28 and 0.9 step 46pt with {\arrow{>}},
}}}}

\begin{document}

\begin{frontmatter}[classification=text]
%% EDITOR: this will force the keywords to appear right after the Abstract.
%%   If the abstract is too long and would force the keywords off the
%%   front page, please comment out % [classification=text] above
%%   This way the keywords will be floated on the bottom of the first page
%%   even though the Abstract spills over to the next page.

%%% AUTHOR: Title goes here.  This line is optional.  You must use it
%%   if title has footnote attached or requires nontrivial typesetting,
%%   e.g., inclusion of linebreaks to force nice layout.
\title{Decomposing a Signed Graph into Rooted Circuits\let\thefootnote\relax\footnotetext{This version contains a correction to Lemma~\ref{lem:main} which adds an additional assumption about the half-edge~$h$.}\thefootnote} %% please capitalize all significant words

%%% AUTHOR:
%%% List all authors. If you wish, place grant acknowledgements in \thanks.
%%% In brackets include a short tag for each author.
\author[rm]{Rose McCarty\thanks{Supported by the National Science Foundation under Grant No. DMS-2202961.}}

%%% AUTHOR: Abstract goes here
\begin{abstract}
We prove a precise min-max theorem for the following problem. Let $G$ be an Eulerian graph with a specified set of edges $S \subseteq E(G)$, and let $b$ be a vertex of $G$. Then what is the maximum integer $k$ so that the edge-set of $G$ can be partitioned into $k$ non-zero $b$-trails? That is, each trail must begin and end at $b$ and contain an odd number of edges from~$S$.
    
This theorem is motivated by a connection to vertex-minors and yields two conjectures of M\'{a}\v{c}ajov\'{a} and \v{S}koviera as corollaries.
\end{abstract}
\end{frontmatter}

%%% AUTHOR: body of paper starts here
\section{Introduction}

We prove a precise min-max theorem (stated as Theorem~\ref{thm:mainFlooding}) for the following problem. We consider finite graphs that are allowed to have loops and multiple edges. Informally, a \textit{signed graph} is a graph whose edges are labelled by the two element group $\mathbb{Z}_2$. (It is more standard to specify a signed graph by its set of edges with label~$1$, but this other formulation is more convenient for us.) A \textit{circuit} is a closed trail, that is, a walk which begins and ends at the same vertex and has no repeated edges (though it may visit vertices multiple times). A circuit \textit{hits} a vertex $b$ if it has an edge incident to $b$, and is \textit{non-zero} if the sum of the labels of its edges is~$1$ (instead of~$0$). A \emph{circuit-decomposition} is a collection of circuits so that each edge of the graph is used by exactly one circuit in the collection.

\begin{problem}
\label{problem:flooding}
Given a signed graph with a vertex $b$, what is the maximum size of a circuit-decomposition where each circuit is non-zero and hits~$b$? 
\end{problem}

\noindent If there is no such circuit-decomposition (for instance if the graph is not Eulerian), then we consider the maximum to be~$0$. 

A major motivation for Problem~\ref{problem:flooding} is a connection with vertex-minors due to Bouchet~\cite{BouchetCircleChar} and Kotzig~\cite{KotzigImmersions}. Roughly, the \textit{vertex-minors} of a simple graph $G$ are the graphs that can be obtained from $G$ by performing ``local complementations at vertices'' (that is, by replacing the induced subgraph on the neighborhood of a vertex by its complement) and by deleting vertices. Vertex-minors were discovered by Bouchet~\cite{isoSystems, graphicIsoSystems} through his work on isotropic systems and have since found many applications. In particular, they have been used to characterize circle graphs~\cite{BouchetCircleChar} and classes of bounded rank-width~\cite{gridThmVM} and rank-depth~\cite{shrubVM}.

Oum~\cite[Question~6]{RWSurvey} conjectures that \emph{every} graph class which is closed under vertex-minors and isomorphism can be characterized by finitely many obstructions. This is true for classes of circle graphs (using well-quasi-ordering for immersion minors~\cite{graphMinors20WQOImmersions} and the work of Kotzig~\cite{KotzigImmersions}) and for classes of bounded rank-width~\cite{RWandWQO}. Motivated by Oum's conjecture, Geelen conjectures that every proper vertex-minor-closed class has a very simple structure; this conjecture is analogous to the Graph Minors Structure Theorem of Robertson and Seymour~\cite{graphMinors16Structure}, but for vertex-minors instead of minors. The conjectured structure is based on the known cases mentioned above; see~\cite{McCartyThesis} for a formal statement. 

We believe that our main theorem, Theorem~\ref{thm:mainFlooding}, will be useful in solving these conjectures of Oum and Geelen. Unfortunately, we also believe that we will need a more general theorem about decomposing graphs whose edges are labelled by the group $\mathbb{Z}_2^k$ (that is, the unique abelian group with $k$ generators, all of order~$2$). In joint work with Jim Geelen and Paul Wollan, we believe that we can prove such a min-max theorem, even for arbitrary groups. That paper is in preparation, but we present the case of signed graphs separately for several reasons. First of all, the other proof (and even the other theorem statement) are considerably more technical for general groups. 

Secondly, the proof we present here for signed graphs relies on finding a matroid underlying Problem~\ref{problem:flooding}. While we believe that there is also a matroid underlying the group-labelled version of Problem~\ref{problem:flooding}, we do not know how to use that matroid to prove a min-max theorem for any other group. It is tempting to hope that this connection might eventually lead to a common generalization with the theorem of Chudnovsky, Geelen, Gerards, Goddyn, Lohman, and Seymour~\cite{ChudnovskyAPaths}; they used a similar matroid-based approach to prove a min-max theorem for vertex-disjoint non-zero ``rooted'' paths in group-labelled graphs. Their proof is in turn based on a short proof of the Tutte–Berge Formula using the matching matroid (see~\cite{ChudnovskyAPaths}). 

Moreover, the case of signed graphs lets us prove two conjectures of M\'{a}\v{c}ajov\'{a} and \v{S}koviera~\cite{MacajovaOddEul} as corollaries. The first of these corollaries is particularly interesting because it does not have a fixed ``root'' vertex.

\begin{restatable}{corollary}{corMS}
\label{cor:corMS}
For any positive integer $\ell$ and any connected $2\ell$-regular graph with an odd number of vertices, there exists a circuit-decomposition of size~$\ell$ where all circuits have an odd number of edges and begin and end at the same vertex. 
\end{restatable}
\noindent We prove Corollary~\ref{cor:corMS} by reducing it to (a very slight strengthening of) the other conjecture of  M\'{a}\v{c}ajov\'{a} and \v{S}koviera from~\cite{MacajovaOddEul}. That other corollary does have a fixed root vertex and follows directly from our min-max theorem.

Finally, we use the min-max theorem to relate Problem~\ref{problem:flooding} to well-known ``packing problems'' for signed graphs. These problems ask for a collection of edge-disjoint circuits instead of a circuit-decomposition. Such problems have been particularly well-studied in relation to the Erd\H{o}s-P\'{o}sa Property; see~\cite{ChurchleyPacking, packingRooted, ErdosPosa4EdgeConn, ReedEscherWall}. Furthermore, Churchley~\cite[Lemma~3.5]{Churchley17} observed that a min-max theorem for the packing version of Problem~\ref{problem:flooding} follows from the theorem in~\cite{ChudnovskyAPaths} mentioned above. The following corollary of Theorem~\ref{thm:mainFlooding} shows that ``packing'' and ``decomposing'' are related when the graph has an Eulerian circuit and a little edge-connectivity.

\begin{restatable}{corollary}{corPacking}
\label{cor:packingCovering}
For any signed $4$-edge-connected Eulerian graph and any vertex $b$, if there is a collection of $\ell$ edge-disjoint non-zero circuits which hit $b$, then there is a circuit decomposition of size $\lceil \ell /2 \rceil$ where each circuit is non-zero and hits $b$.
\end{restatable}

\noindent The bound is best possible, and $4$-edge-connectivity is necessary. 

This paper is adapted from Chapter~4 of the author's PhD thesis~\cite{McCartyThesis}. In Section~\ref{sec:minMax} we give some important definitions and state the min-max theorem (Theorem~\ref{thm:mainFlooding}). In Section~\ref{sec:matroid} we define the matroid and prove that it is, in fact, a matroid. Finally, in Section~\ref{sec:reduction} we complete the proof of Theorem~\ref{thm:mainFlooding}, and in Section~\ref{sec:cor} we prove its corollaries.

\section{The min-max theorem}
\label{sec:minMax}
In this section we give some preliminary definitions, state the min-max theorem (Theorem~\ref{thm:mainFlooding}), and outline its proof.

\subsection*{Preliminaries}
We use standard graph-theoretic notation; see Diestel~\cite{diestelBook}. For a graph $G$ with a set of vertices $X$, we write $E(X)$ (respectively $\delta(X)$) for the set of edges of $G$ with both ends (respectively, exactly one end) in $X$. We write $G-X$ for the induced subgraph of $G$ on vertex-set $V(G)-X$. If $v$ is a vertex of $G$, then we write $G-v$ for $G-\{v\}$ and $\deg(v)$ for the degree of~$v$. 

We think of graphs as having half-edges; this formulation is used to resolve technical issues with loops. So an \emph{edge} is an unordered pair of half-edges and an \emph{arc} is an ordered pair of half-edges. (It is convenient to use arcs so that trails have a defined ``beginning'' and ``end''.) Thus an edge $\{h_1, h_2\}$ has two corresponding arcs, $(h_1, h_2)$ and $(h_2, h_1)$. The \emph{tail} (respectively \emph{head}) of an arc $(h_1, h_2)$ is the vertex that is incident to $h_1$ (respectively $h_2$). A \emph{trail} is then a sequence of arcs so that the corresponding edges are all distinct and the head of each arc, other than the last one, is the tail of the next. The \emph{tail} of a trail is the tail of its first arc, and the \emph{head} of a trail is the head of its last arc. 

A \emph{subtrail} of a trail $T$ is any trail which can be obtained from $T$ by deleting zero or more arcs at its beginning and end. A \emph{circuit} is a trail which has the same head and tail. If $C$ is a circuit whose sequence of arcs is $a_1, \ldots, a_t$, then we say that any circuit of the form $a_i, a_{i+1}, \ldots, a_t, a_1, a_2, \ldots, a_{i-1}$ is \emph{obtained from $C$ by cyclically re-ordering its arcs}. Thus a circuit \emph{hits} a vertex $v$ if it can be cyclically re-ordered so as to have $v$ as its tail and head (or, equivalently, if it contains an arc which is incident to $v$). A \emph{circuit-decomposition} is a collection of circuits so that each edge of the graph is used by exactly one circuit in the collection. A graph is \emph{Eulerian} if it is connected and every vertex has even degree (or, equivalently, if it has an Eulerian circuit).

A \textit{signed graph} is a tuple $(G, \gamma)$ so that $G$ is a graph and $\gamma$ is a function from the edge-set of $G$ to the $2$-element group $\mathbb{Z}_2$. The function $\gamma$ is called a \textit{signature of $G$}. (This formulation is non-standard; $\gamma$ is typically specified by the set of edges of $G$ which are sent to $1$. However we find this functional formulation more convenient for our purposes.) Given a signed graph $(G, \gamma)$, the \textit{weight} of an edge $e$ is the corresponding group element $\gamma(e)$. The \textit{weight} of a trail $T$ is the sum (in $\mathbb{Z}_2$) of the weights of the edges of $T$; it is denoted by $\gamma(T)$. An edge or a trail is called \textit{zero} or \textit{non-zero} depending on its weight.

%A \textit{trail} is a walk with no repeated edges, and a \textit{circuit} is a closed trail (that is, a trail which begins and ends at the same vertex). A trail \textit{hits} a vertex $v$ if it contains an edge incident to $v$. A \emph{circuit-decomposition} is a collection of circuits so that each edge of $G$ is used by exactly one circuit in the collection.%See if really need definition of 'ends' of a trail

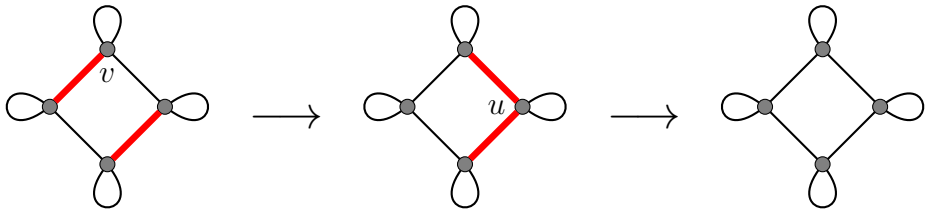
\begin{figure}
\centering
\begin{tikzpicture}[scale =.85, every node/.style={MyNode}]
    %4-regular
    \def \dif {2*.45}
    \def \loop {35}%angle loop leaves and enters
    \def \width {.75}
    %for spacing
    \node[draw=none, fill=none] (Y0) at (-1.5*\width, 0) {};
    \node[draw=none, fill=none] (Y3) at (1.5*\width, 0) {};
    \draw[draw=none, fill=none] (Y0) to [out={270+90*3-\loop},in={270+90*3+\loop},looseness=18] (Y0);
    \draw[draw=none, fill=none] (Y3) to [out={270+90-\loop},in={270+90+\loop},looseness=18] (Y3);
    %nodes
    \node (X0) at (0, -\dif) {};
    \node (X1) at (\dif, 0) {};
    \node[label=below:{$v$}] (X2) at (0, \dif) {};
    \node (X3) at (-\dif, 0) {};
    %for spacing
    \node[draw=none, fill=none] at (0, -2*.75) {};
    \node[draw=none, fill=none] at (0, 2*.75) {};
    %edges
    \draw[MyRedArc] (X0) -- (X1);
    \draw[thick] (X1) -- (X2);
    \draw[MyRedArc] (X2) -- (X3);
    \draw[thick] (X3) -- (X0);
    \foreach \i in {0,..., 3}{
        \draw[thick] (X\i) to [out={270+90*\i-\loop},in={270+90*\i+\loop},looseness=18] (X\i);
    }
\end{tikzpicture}\hskip .125cm
\begin{tikzpicture}[scale=.85, every node/.style={MyNode}]
    \def \r {2}
    %To get the vertical spacing right we add nodes:
    \node[draw=none, fill=none] at (90:\r) {};
    \node[rectangle, draw=none, fill=none] at (0, -2*.75) {};
    \node[rectangle, draw=none, fill=none] (center) at (0, 0) {\Large$\longrightarrow$};
\end{tikzpicture}\hskip .125cm
\begin{tikzpicture}[scale =.85, every node/.style={MyNode}]
    %4-regular
    \def \dif {2*.45}
    \def \loop {35}%angle loop leaves and enters
    \def \width {.75}
    %for spacing
    \node[draw=none, fill=none] (Y0) at (-1.5*\width, 0) {};
    \node[draw=none, fill=none] (Y3) at (1.5*\width, 0) {};
    \draw[draw=none, fill=none] (Y0) to [out={270+90*3-\loop},in={270+90*3+\loop},looseness=18] (Y0);
    \draw[draw=none, fill=none] (Y3) to [out={270+90-\loop},in={270+90+\loop},looseness=18] (Y3);
    %nodes
    \node (X0) at (0, -\dif) {};
    \node[label=left:{$u$}] (X1) at (\dif, 0) {};
    \node (X2) at (0, \dif) {};
    \node (X3) at (-\dif, 0) {};
    %for spacing
    \node[draw=none, fill=none] at (0, -2*.75) {};
    \node[draw=none, fill=none] at (0, 2*.75) {};
    %edges
    \draw[MyRedArc] (X0) -- (X1);
    \draw[MyRedArc] (X1) -- (X2);
    \draw[thick] (X2) -- (X3);
    \draw[thick] (X3) -- (X0);
    \foreach \i in {0,..., 3}{
        \draw[thick] (X\i) to [out={270+90*\i-\loop},in={270+90*\i+\loop},looseness=18] (X\i);
    }
\end{tikzpicture}\hskip .125cm
\begin{tikzpicture}[scale=.85, every node/.style={MyNode}]
    \def \r {2}
    %To get the vertical spacing right we add nodes:
    \node[draw=none, fill=none] at (90:\r) {};
    \node[rectangle, draw=none, fill=none] at (0, -2*.75) {};
    \node[rectangle, draw=none, fill=none] (center) at (0, 0) {\Large$\longrightarrow$};
\end{tikzpicture}\hskip .125cm
\begin{tikzpicture}[scale =.85, every node/.style={MyNode}]
    %4-regular
    \def \dif {2*.45}
    \def \loop {35}%angle loop leaves and enters
    \def \width {.75}
    %for spacing
    \node[draw=none, fill=none] (Y0) at (-1.5*\width, 0) {};
    \node[draw=none, fill=none] (Y3) at (1.5*\width, 0) {};
    \draw[draw=none, fill=none] (Y0) to [out={270+90*3-\loop},in={270+90*3+\loop},looseness=18] (Y0);
    \draw[draw=none, fill=none] (Y3) to [out={270+90-\loop},in={270+90+\loop},looseness=18] (Y3);
    %nodes
    \node (X0) at (0, -\dif) {};
    \node (X1) at (\dif, 0) {};
    \node (X2) at (0, \dif) {};
    \node (X3) at (-\dif, 0) {};
    %for spacing
    \node[draw=none, fill=none] at (0, -2*.75) {};
    \node[draw=none, fill=none] at (0, 2*.75) {};
    %edges
    \draw[thick] (X0) -- (X1);
    \draw[thick] (X1) -- (X2);
    \draw[thick] (X2) -- (X3);
    \draw[thick] (X3) -- (X0);
    \foreach \i in {0,..., 3}{
        \draw[thick] (X\i) to [out={270+90*\i-\loop},in={270+90*\i+\loop},looseness=18] (X\i);
    }
\end{tikzpicture}
\caption{Shifting a signature at a vertex $v$ and then a vertex $u$. Throughout the paper, non-zero edges are depicted in bold red.}
\label{fig:shifting}
\end{figure}

Signatures are only used to specify which circuits of a graph are zero/non-zero; so there is an equivalence relation on signatures as follows. First, \emph{shifting at} a vertex means to add~$1$ to the weight of each incident non-loop edge; see Figure~\ref{fig:shifting} for an example. (We note that this operation is more commonly called ``switching'', as in the survey on signed graphs by Zaslavsky~\cite{ZaslavskySignedMatroid}. We opt to use the word ``shifting'' just for consistency with the setting of group-labelled graphs, as in~\cite{ChudnovskyAPaths}.) A \emph{shifting} of a signature $\gamma$ is any signature that can be obtained from $\gamma$ by performing a sequence of shiftings at vertices. Equivalently, a shifting is obtained from $\gamma$ by adding $1$ to the weight of each edge in a cut. Thus shifting is an equivalence relation that does not change the weight of any circuit. (Harary~\cite{HararyBalanced} also proved a converse; if each circuit of a graph has the same weight according to two signatures, then the signatures are shiftings of each other.)

Throughout the paper we are interested in ``rooted Eulerian signed'' graphs; so we call an \emph{\rooted{}} a tuple $(G, \gamma, b)$ so that $G$ is an Eulerian graph, $\gamma$ is a signature of $G$, and $b$ is a vertex of $G$. We call $b$ the \emph{root} of $(G, \gamma, b)$. We write $\tilde{\nu}(G, \gamma, b)$ for the answer to Problem~\ref{problem:flooding}; that is, $\tilde{\nu}(G, \gamma, b)$ is the maximum size of a circuit-decomposition where each circuit is non-zero and hits $b$. We call $\tilde{\nu}(G, \gamma, b)$ the \emph{flooding number} of $(G, \gamma, b)$, and consider it to be zero if no such circuit-decomposition exists.

\subsection*{Theorem statement}
Let us consider why an \rooted{} $(G, \gamma, b)$ might have small flooding number.

One reason is that, after shifting, there is a small edge-cut so that the side containing $b$ has few non-zero edges. Formally, for a set of edges $F$ and a shifting $\gamma'$ of $\gamma$, we write $\gamma'(F)$ for the number of non-zero edges in $F$ according to $\gamma'$. Using this notation we can state the following upper bound;  \begin{align}
    \label{eqn:notTight}
    \tilde{\nu}(G, \gamma, b) \leq \min_{\gamma', X}\left(\gamma'(E(X))+\frac{1}{2}|\delta(X)|\right),
\end{align} where the minimum is taken over all shiftings $\gamma'$ of $\gamma$ and all sets of vertices $X$ which contain~$b$.

If we did not require the edge-sets of the circuits to partition the edge-set of $G$, but just to be disjoint, then inequality~(\ref{eqn:notTight}) would be tight; this fact was observed by Churchley~\cite[Lemma~3.5]{Churchley17} following from~\cite{ChudnovskyAPaths}. For the flooding number, however, inequality~(\ref{eqn:notTight}) is not tight. Intuitively, this is because parity matters; since we are interested in circuit-decompositions, the flooding number must have the same parity as $\gamma(E(G))$. So in Figure~\ref{fig:floodingDiff}, for instance, the flooding number must be odd. Therefore, while that example has~$\deg(b)/2=4$ edge-disjoint non-zero circuits which hit $b$, its flooding number is just three.

\begin{figure}
\centering
\begin{tikzpicture}[scale=1, every node/.style={MyNode}]
    \def \w {2}
    \node (A1) at (0,0) {};
    \node (A2) at (\w,0) {};
    \node (A3) at (2*\w,0) {};
    \node (A4) at (3*\w,0) {};
    \node[label=below:$b$] (a) at (1.5*\w,-1.5) {};
    \draw[thick] (A1) to (A2);
    \draw[thick] (A2) to (A3);
    \draw[thick] (A3) to (A4);
    \draw[MyRedArc] (A1) to [bend left=25] (A4);
    \foreach \i in {1,4}{
        \draw[MyRedArc] (a) to [bend left=12] (A\i);
        \draw[thick] (a) to [bend right=12] (A\i);
    }\foreach \i in {2,3}{
        \draw[MyRedArc] (a) to [bend left=18] (A\i);
        \draw[thick] (a) to [bend right=18] (A\i);
    }
        %\draw[thick] (A1\i) to [bend right=12] (A3\i);

\end{tikzpicture}
\caption{An \rooted{} with four edge-disjoint non-zero circuits which hit $b$, but with flooding number three.}
\label{fig:floodingDiff}
\end{figure}
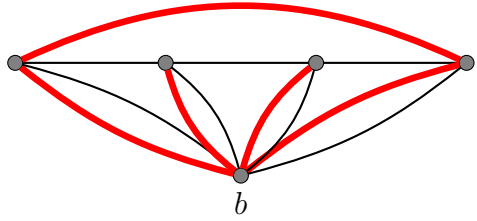

It turns out that ``parity'' is the only possible problem; the min-max theorem says that if we subtract one for each component of $G-X$ where ``the parity is wrong'', then inequality~(\ref{eqn:notTight}) becomes tight. To state this formally, let $(G, \gamma, b)$ be an \rooted{}, and let $\gamma'$ be a shifting of $\gamma$. A set of vertices $Y$ is \emph{$\gamma'$-odd} if the parity of $\gamma'(E(Y) \cup \delta(Y))$ is different from the parity of $|\delta(Y)|/2$. Then, for a set of vertices $X$ which contains $b$, we write $\mathrm{odd}_{\gamma'}(G-X)$ for the number of components of $G-X$ whose vertex-set is $\gamma'$-odd. Now we can state the theorem.

\begin{restatable}{theorem}{thmMain}
\label{thm:mainFlooding}
For any \rooted{} $(G, \gamma,b)$,\begin{align*}
    \tilde{\nu}(G, \gamma, b) = \min_{\gamma', X} \left(\gamma'(E(X))+\frac{1}{2}|\delta(X)|-{\mathrm{odd}}_{\gamma'}(G-X)\right),
\end{align*}
where the minimum is taken over all shiftings $\gamma'$ of $\gamma$ and all sets of vertices $X$ which contain~$b$.
\end{restatable}

We go ahead and prove the easy direction of Theorem~\ref{thm:mainFlooding} now: that the right-hand side of the equation is an upper-bound for the flooding number.

\begin{lemma}
\label{lem:easyDirection}
For any \rooted{} $(G, \gamma, b)$, shifting $\gamma'$ of $\gamma$, and set of vertices $X$ which contains $b$,\begin{align}
    \label{inequality}
    \tilde{\nu}(G, \gamma, b) \leq \gamma'(E(X))+\frac{1}{2}|\delta(X)|-\mathrm{odd}_{\gamma'}(G-X).
\end{align}
\end{lemma}
\begin{proof}
Since shifting does not change the flooding number, we may assume that $\gamma'=\gamma$. Now let $\C$ be a circuit-decomposition of size $\tilde{\nu}(G, \gamma, b)$ so that each circuit in $\C$ is non-zero and hits $b$. We may assume that each circuit in $\C$ has $b$ as its tail and head by cyclically re-ordering it. 

We now ``split up'' the circuits in $\C$ into a collection $\mathcal{T}$ of edge-disjoint trails. First add a trail to~$\mathcal{T}$ for each edge in $E(X)$, oriented the same way as in $\C$. Then add each trail $T$ which satisfies the following conditions to~$\mathcal{T}$:\begin{enumerate}
    \item $T$ is a subtrail of a circuit in $\C$,
    \item the tail and head of $T$ are in $X$, and
    \item the first and last edges of $T$ are in $\delta(X)$, and no other edges of $T$ are.
\end{enumerate}
\noindent Observe that $\mathcal{T}$ has size $|E(X)|+|\delta(X)|/2$, and that the edge-sets of the trails in $\mathcal{T}$ partition $E(G)$. We now show that the number of non-zero trails in $\mathcal{T}$ is at most the right-hand side of inequality~(\ref{inequality}); this will complete the proof of Lemma~\ref{lem:easyDirection} since each circuit in $\C$ contributes at least one non-zero trail to $\mathcal{T}$.

It is enough to show that, for each component of $G-X$ whose vertex-set $Y$ is $\gamma$-odd, there exists a trail in $\mathcal{T}$ which has weight zero and hits a vertex in $Y$. There are $|\delta(Y)|/2$ trails in $\mathcal{T}$ which hit a vertex in $Y$. Moreover, the sum of their weights (in $\mathbb{Z}_2$) is equal to the parity of $\gamma(E(Y)\cup \delta(Y))$. So at least one of these trails has weight zero since, by the definition of odd components, the parity of $\gamma(E(Y)\cup \delta(Y))$ is different from the parity of $|\delta(Y)|/2$. 
\end{proof}

\subsection*{Proof outline} We now outline the proof of Theorem~\ref{thm:mainFlooding}. 

Let $(G, \gamma, b)$ be an \rooted{}. A \emph{certificate} for $(G, \gamma, b)$ is a tuple $(X, \gamma')$ so that $\gamma'$ is a shifting of $\gamma$, $X$ is a set of vertices that contains $b$, and inequality~(\ref{inequality}) from Lemma~\ref{lem:easyDirection} is tight. Thus, to prove Theorem~\ref{thm:mainFlooding}, we just need to show that every \rooted{} has a certificate. 

The key observation is that we can allow circuit-decompositions to contain zero circuits. Formally, a \emph{flooding} is a circuit-decomposition of size $\deg(b)/2$ where each circuit has $b$ as its tail and head. (This assumption is simply more convenient than saying that $b$ is hit.) A flooding is \emph{optimal} if it contains the maximum number of non-zero circuits in any flooding of $(G, \gamma, b)$. This gives an alternate definition of the flooding number as follows.

\begin{lemma}
\label{lem:floodingEq}
For any \rooted{} $(G, \gamma, b)$, the maximum number of non-zero circuits in a flooding is equal to $\tilde{\nu}(G, \gamma, b)$.
\end{lemma}
\begin{proof}
First consider a circuit-decomposition $\C$ of size $\tilde{\nu}(G, \gamma, b)$ so that each circuit in $\C$ is non-zero and hits $b$. By ``splitting up the circuits in $\C$ at $b$'', we can find a flooding that contains at least $\tilde{\nu}(G, \gamma, b)$-many non-zero circuits. (Each non-zero circuit yields an odd number of non-zero circuits after ``splitting''.) 

In the other direction, consider an optimal flooding $\C$. All of the zero circuits in $\C$ can be ``combined'' with a non-zero circuit in $\C$ to obtain a circuit-decomposition where each circuit is non-zero and hits $b$. (We may assume that $\C$ contains a non-zero circuit since otherwise this direction holds trivially.)
\end{proof}
\noindent We work with this alternate definition of the flooding number from now on.

Informally, our approach to Theorem~\ref{thm:mainFlooding} is to consider why the zero circuits in an optimal flooding cannot be ``turned into'' non-zero circuits. Every optimal flooding contains exactly $\left(\deg(b)/2-\tilde{\nu}(G, \gamma, b)\right)$-many zero circuits. In Section~\ref{sec:matroid} we define a matroid of this rank. The bases of this matroid are obtained by selecting a ``representative'' for each zero circuit in an optimal flooding. Roughly, a representative is specified by 1) a ``split'' of the circuit into two subtrails, one with $b$ as its tail and one with $b$ as its head, and 2) the weight of those two subtrails (they must have the same weight since the circuit has weight zero). 

Then in Section~\ref{sec:reduction} we prove Theorem~\ref{thm:mainFlooding} by reducing to the case that this matroid has rank~$1$. Thus $\tilde{\nu}(G, \gamma, b)=\deg(b)/2-1$, and ``parity'' already shows that $(\{b\}, \gamma)$ is a certificate. We discuss this reduction in a little more detail in the next section, after defining the matroid.

For this approach, it is convenient to have some notation about how to ``combine'' and ``split up'' trails. If $T_1$ and $T_2$ are edge-disjoint trails so that the head of $T_1$ is the tail of $T_2$, then we can compose them into a new trail denoted $(T_1,T_2)$. If $\gamma$ is a signature, then we denote the weight of $(T_1, T_2)$ by $\gamma(T_1, T_2)$ for short. Likewise, we can reverse a trail $T$ to obtain another trail denoted $T^{-1}$. We also use this notation if $f$ is an arc; so $f^{-1}$ is the arc with the same edge, but in the reverse direction. As an example of this notation, we have $(T_1,T_2)^{-1} = (T_2^{-1},T_1^{-1})$. 

We use transitions to ``split up'' trails; a \emph{transition} is a set of two half-edges which are incident to the same vertex. If that vertex is $v$, we say the transition is \emph{at $v$}. The \emph{transitions of a trail} $T$ are the transitions $\{h_1, h_2\}$ so that $T$ has two consecutive arcs of the form $(h_1', h_1)$ and $(h_2, h_2')$. So a trail with $\ell$ arcs is fully determined by its first arc and its $\ell-1$ transitions. The \emph{transitions of a flooding} are the transitions of its circuits. Finally, if $R_1, \ldots, R_k$ are distinct transitions of a trail $T$, then there are unique non-empty trails $T_1, \ldots, T_{k+1}$ so that $T = (T_1, \ldots, T_{k+1})$ and none of $R_1, \ldots, R_k$ are transitions of any of $T_1, \ldots, T_{k+1}$. We say that $(T_1, \ldots, T_{k+1})$ is the \emph{split of $T$ specified by $R_1, \ldots, R_k$}. If $\mathcal{T}$ is a collection of edge-disjoint trails and $R_1, \ldots, R_k$ are transitions of a single trail $T \in \mathcal{T}$, then we also call $(T_1, \ldots, T_{k+1})$ the \emph{split of $\mathcal{T}$ specified by $R_1, \ldots, R_k$}.

%THE MATROID
\section{The flooding matroid}
\label{sec:matroid}

Let $(G, \gamma, b)$ be an \rooted{}, and let $C$ be a zero circuit which has $b$ as its tail and head. A \emph{representative for $C$} is a tuple $(f, \alpha)$ so that $f$ is an arc of $C$ and $\alpha \in \{0,1\}$ is the weight of the subtrail of $C$ which is obtained by deleting all arcs after $f$; see Figure~\ref{fig:representative}. A \emph{system of representatives} for a flooding $\C$ is a set $B$ that consists of one representative for each zero circuit in $\C$. 

We define the \emph{flooding matroid} $M(G, \gamma, b)$ by its ground set and its bases. While we call it a matroid right away, it will take us the rest of this section to prove that it is in fact a matroid. So for now it should just be viewed as a combinatorial object which consists of a ground set and some collection of subsets of that ground set called ``bases''. Here is the definition. The ground set of $M(G, \gamma, b)$ is the set of all tuples $(f, \alpha)$ so that $f$ is an arc of $(G, \gamma, b)$ and $\alpha \in \{0,1\}$. A set $B$ is a basis of $M(G, \gamma, b)$ if it is a system of representatives for an optimal flooding. (If the flooding number is equal to $\deg(b)/2$, then we view the empty set as a system of representatives for an optimal flooding; this guarantees that $M(G, \gamma, b)$ always has a basis.)

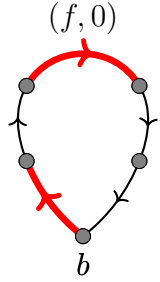
\begin{figure}
\centering
\begin{tikzpicture}[use Hobby shortcut, scale=1, every node/.style={MyNode}]
    \def \w {2}
    %Draw A
    \node[label=below:$b$] (a) at (0,0) {};
    \node (A1) at (-.75,1) {};
    \node (A2) at (-.75,2) {};
    \node (A3) at (.75,2) {};
    \node (A4) at (.75,1) {};
    \draw[thick, arrBlack] (a.center) .. ([blank]A1.center) .. (A2.center) .. ([blank]A3.center) .. (A4.center) .. (a.center);
    \draw[MyRedArc, arrRed] (a.center) .. (A1.center) .. ([blank]A2.center) .. (A3.center) .. ([blank]A4.center) .. ([blank]a.center);
    \foreach \i in {1,...,4}{
        \node at (A\i) {};
    }
    \node[label=below:$b$] (a) at (0,0) {};
    \node[rectangle, draw=none, fill=none] at (0,2.9) {$(f, 0)$};
\end{tikzpicture}
\caption{A circuit which is represented by $(f,0)$, where $f$ is the third arc of the circuit.}
\label{fig:representative}
\end{figure}

Recall that, to prove Theorem~\ref{thm:mainFlooding}, we will reduce to the case that $M(G, \gamma, b)$ has rank~$1$. The key step is show that if $(G, \gamma, b)$ is a counterexample to Theorem~\ref{thm:mainFlooding} which is, in a certain sense, ``minimal'', then for each arc $f$ of $G-b$, both $(f,0)$ and $(f,1)$ are non-loop elements of $M(G, \gamma, b)$. Then we will use the transitivity of parallel pairs and the following key lemma. The proof of the lemma does not use the fact that $M(G, \gamma, b)$ is a matroid, just the definition of its bases.

\begin{lemma}
\label{lem:transitivity}
If $(G, \gamma, b)$ is an \rooted{} and $f_0$ and $f_1$ are arcs with the same head, then there is no basis of $M(G, \gamma, b)$ which contains both $(f_0, 0)$ and $(f_1, 1)$.
\end{lemma}
\begin{proof}
Suppose to the contrary that there is such a basis. Then there exists an optimal flooding $\C$ that contains distinct zero circuits with $(f_0,0)$ and $(f_1,1)$ as their respective representatives. Thus there are trails $T_0, S_0, T_1, S_1$ so that $(T_0, S_0)$ and $(T_1, S_1)$ are distinct zero circuits in $\C$, the trail $T_0$ has weight~$0$, the trail $T_1$ has weight~$1$, and $T_0$ and $T_1$ have the same head. We can obtain another flooding $\C'$ from $\C$ by replacing $(T_0, S_0)$ and $(T_1, S_1)$ with the circuits $(T_0, T_1^{-1})$ and $(S_0^{-1}, S_1)$. However this contradicts the optimality of $\C$ as the two new circuits are both non-zero.
\end{proof}

Let $(G, \gamma, b)$ be an \rooted{}. The rest of this section is dedicated to proving that $M(G, \gamma, b)$ is a matroid. To do so, we will prove that the basis exchange axiom holds in Lemma~\ref{lem:basisExchange}. The proof reduces to the $4$-edge-connected case; $(G, \gamma, b)$ is \textit{$4$-edge-connected} if there is no set $Y \subseteq V(G)-\{b\}$ so that $|\delta(Y)|=2$. Then we use the following key lemma to find a transition which works for two different bases; a transition $R$ \textit{works for} a basis $B$ of $M(G, \gamma, b)$ if there exists an optimal flooding which has $R$ as a transition and $B$ as a system of representatives.

\begin{lemma}
\label{lem:main}
Let $(G, \gamma, b)$ be a $4$-edge-connected \rooted{}, let $v\neq b$ be a vertex which is joined to $b$ by an edge $e$, and let $h$ be the half-edge of $e$ which is at $v$. Then for any basis $B$ of $M(G, \gamma, b)$, more than half of the transitions at $v$ which contain $h$ work for $B$.
\end{lemma}
\begin{proof}
Say that a transition is \emph{valid} if it is a transition at $v$ that contains $h$. So we are trying to show that more than half of the valid transitions work for~$B$.

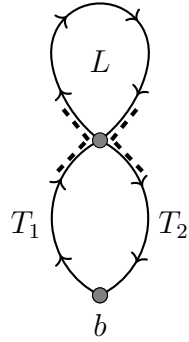
\begin{figure}
\centering
\begin{tikzpicture}[use Hobby shortcut, scale=.75]
    \def \w {1.95}
    %Draw A
    \node (v) at (0,0) {};
    \node (midL) at (-.85,-1.5) {};
    \node (midR) at (.85,-1.5) {};
    \node[label=left:$T_1$] at (-.65,-1.5) {};
    \node[label=right:$T_2$] at (.65,-1.5) {};
    \node[label=below:$b$, MyNode] (a) at (0,-2.75) {};
    \node (A1) at (-.75,1) {};
    \node (A2) at (-.75,2) {};
    \node (A3) at (.75,2) {};
    \node (A4) at (.75,1) {};
    \draw[ultra thick, dashed] (-.75, -.55) -- (-.23, 0);
    \draw[ultra thick, dashed] (.75, -.55) -- (.23, 0);
    \draw[ultra thick, dashed] (-.65,.55) -- (-.2, 0);
    \draw[ultra thick, dashed] (.65,.55) -- (.2, 0);
    \draw[thick, arrLoop] (v.center) .. (A1.center) .. (A2.center) .. (A3.center) .. (A4.center) .. (v.center);
    \draw[thick, arrLoopBot] (a.center) .. (midL.center) .. (v.center);
    \draw[thick, arrLoopBotReverse] (a.center) .. (midR.center) .. (v.center);
    \foreach \i in {1,...,4}{
        \node at (A\i) {};
    }
    \node[MyNode] at (v) {};
    \node[MyNode] at (a) {};
    \node[rectangle] at (0,1.4) {$L$};
\end{tikzpicture}
\caption{The transitions $\{r, r'\}$ and $\{h,h'\}$ and the split $(T_1, L, T_2)$ from \textit{Case~1} of Lemma~\ref{lem:main}. Throughout the paper, transitions are depicted as thick dashed curves.}
\label{fig:mainLemma1}
\end{figure}

Fix an optimal flooding $\C$ which has $B$ as a system of representatives. There is a unique half-edge $h'$ so that $\{h,h'\}$ is a transition of $\C$. So $\{h,h'\}$ works for $B$. Furthermore, there are exactly $\deg(v)-1$ valid transitions, and $\deg(v)-1$ is odd. So it suffices to show that half of the other valid transitions also work for $B$. We will do this by proving that for each transition $\{r, r'\} \neq \{h,h'\}$ of $\C$ at $v$, either $\{h, r\}$ or $\{h, r'\}$ works for $B$. (Note that the transitions of $\C$ at $v$ are pairwise disjoint.) We break into two cases.

\smallskip
\noindent\emph{Case 1:} $\{h,h'\}$ and $\{r, r'\}$ are transitions of the same circuit in $\C$.
\smallskip

Let $(T_1, L, T_2)$ be the split of $\C$ specified by $\{h,h'\}$ and $\{r, r'\}$; see Figure~\ref{fig:mainLemma1}. Since $\C$ is a flooding and the edge containing $h$ is incident to $b$, it follows that $h$ is in an arc of $T_1$ or $T_2$ (and not in an arc of $L$). We can obtain a new flooding from $\C$ by replacing $(T_1, L, T_2)$ with either $(T_1, L^{-1}, T_2)$ or its reversal $(T_1, L^{-1}, T_2)^{-1}$. If $B$ is a system of representatives for either of these floodings, then we are done; so we may assume otherwise. It follows that $(T_1, L, T_2)$ is a zero circuit, since basis elements are only taken from zero circuits, and changing the order of arcs in a circuit does not affect its weight. Moreover, the arc of the representative of $(T_1, L, T_2)$ is in $L$, for otherwise the subtrail leading up to that arc in $(T_1, L^{-1}, T_2)$ has the same weight as in $(T_1, L, T_2)$. Finally, note that $\gamma(T_1) \neq \gamma(T_2)$, for otherwise $B$ is a system of representatives for the flooding obtained by replacing $(T_1, L, T_2)$ with $(T_1, L^{-1}, T_2)^{-1}=(T_2^{-2}, L, T_1^{-1})$.

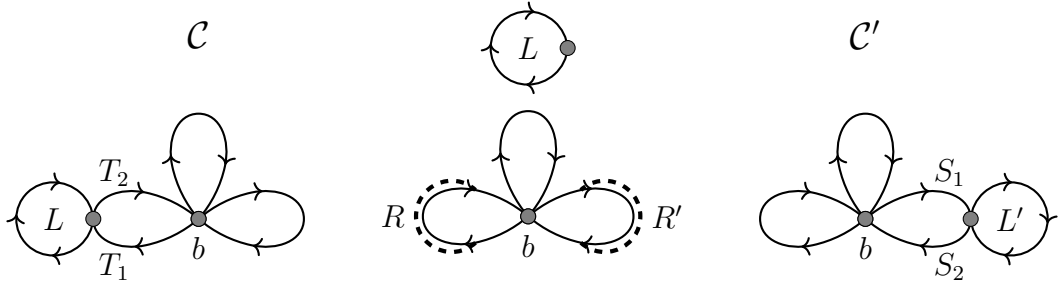
\begin{figure}
\centering
\begin{tikzpicture}[use Hobby shortcut, scale=.7]
    \node (v) at (-2,0) {};
    \node (v1) at (0, 2) {};
    \node (v2) at (2, 0) {};
    \node (midB) at (-1, -.5) {};
    \node (midT) at (-1, .5) {};
    \node (midL) at (-.5, 1) {};
    \node (midR) at (.5, 1) {};
    \node (midB2) at (1, -.5) {};
    \node (midT2) at (1, .5) {};
    \node[label=below:$T_1$] at (-1.6, -.25) {};
    \node[label=above:$T_2$] at (-1.6, .21) {};
    \node[label=below:$b$, MyNode] (a) at (0,0) {};
    \node (A1) at (-2.25, -.5) {};
    \node (A2) at (-3.25, -.5) {};
    \node (A3) at (-3.25, .5) {};
    \node (A4) at (-2.25, .5) {};
    \draw[thick, arrLoopSmaller] (v.center) .. (A1.center) .. (A2.center) .. (A3.center) .. (A4.center) .. (v.center);
    \draw[thick, arrLoopMid] (a.center) .. (midB.center) .. (v.center) .. (midT.center) .. (a.center);
    \draw[thick, arrLoopMid] (a.center) .. (midL.center) .. (v1.center) .. (midR.center) .. (a.center);
    \draw[thick, arrLoopMid] (a.center) .. (midT2.center) .. (v2.center) .. (midB2.center) .. (a.center);
    \node[MyNode] at (a) {};
    \node[MyNode] at (v) {};
    \node[rectangle] at (0, 3.8) {};
    \node[rectangle] at (0, 3.5) {\large $\C$};
    \node[rectangle] at (-2.75,0) {$L$};
\end{tikzpicture}\hskip .75cm
\begin{tikzpicture}[use Hobby shortcut, scale=.7]
    \node (v) at (-2,0) {};
    \node (v1) at (0, 2) {};
    \node (v2) at (2, 0) {};
    \node (midB) at (-1, -.5) {};
    \node (midT) at (-1, .5) {};
    \node (midL) at (-.5, 1) {};
    \node (midR) at (.5, 1) {};
    \node (midB2) at (1, -.5) {};
    \node (midT2) at (1, .5) {};
    \node[label=below:{\textcolor{white}{$T_1$}}] at (-1.4, -.3) {};
    \node[label=below:$b$, MyNode] (a) at (0,0) {};
    \node (A1') at (.5, 2.7) {};
    \node (A2') at (-.5, 2.7) {};
    \node (A3') at (-.5, 3.7) {};
    \node (A4') at (.5, 3.7) {};
    \node (v') at (.75, 3.2) {};
    %The transitions
    \node (vT) at (-2.13,0) {};
    \node (v2T) at (2.13, 0) {};
    \node[label=left:$R$] at (-1.93 ,0) {};
    \node[label=right:$R'$] at (1.93 ,0) {};
    \draw[ultra thick, dashed] (midB.center) .. (vT.center) .. (midT.center);
    \draw[ultra thick, dashed] (midB2.center) .. (v2T.center) .. (midT2.center);
    %The rest
    \draw[thick, arrLoopSmaller] (v'.center) .. (A1'.center) .. (A2'.center) .. (A3'.center) .. (A4'.center) .. (v'.center);
    \draw[thick, arrLoopMid] (a.center) .. (midB.center) .. (v.center) .. (midT.center) .. (a.center);
    \draw[thick, arrLoopMid] (a.center) .. (midL.center) .. (v1.center) .. (midR.center) .. (a.center);
    \draw[thick, arrLoopMid] (a.center) .. (midT2.center) .. (v2.center) .. (midB2.center) .. (a.center);
    \node[MyNode] at (v') {};
    \node[MyNode] at (a) {};
    \node[rectangle] at (0, 3.8) {};
    \node[rectangle] at (0, 3.2) {$L$};
\end{tikzpicture}\hskip .75cm
\begin{tikzpicture}[use Hobby shortcut, scale=.7]
    \node (v) at (-2,0) {};
    \node (v1) at (0, 2) {};
    \node (v2) at (2, 0) {};
    \node (midB) at (-1, -.5) {};
    \node (midT) at (-1, .5) {};
    \node (midL) at (-.5, 1) {};
    \node (midR) at (.5, 1) {};
    \node (midB2) at (1, -.5) {};
    \node (midT2) at (1, .5) {};
    \node[label=below:$S_2$] at (1.6, -.25) {};
    \node[label=above:$S_1$] at (1.6, .21) {};
    \node[label=below:$b$, MyNode] (a) at (0,0) {};
    \node (A1) at (2.25, -.5) {};
    \node (A2) at (3.25, -.5) {};
    \node (A3) at (3.25, .5) {};
    \node (A4) at (2.25, .5) {};
    %\draw[ultra thick, dashed] (-.75, -.65) -- (-.23, 2);
    \draw[thick, arrLoopSmaller] (v2.center) .. (A4.center) .. (A3.center) .. (A2.center) .. (A1.center) .. (v2.center);
    \draw[thick, arrLoopMid] (a.center) .. (midB.center) .. (v.center) .. (midT.center) .. (a.center);
    \draw[thick, arrLoopMid] (a.center) .. (midL.center) .. (v1.center) .. (midR.center) .. (a.center);
    \draw[thick, arrLoopMid] (a.center) .. (midT2.center) .. (v2.center) .. (midB2.center) .. (a.center);
    \node[MyNode] at (a) {};
    \node[MyNode] at (v2) {};
    \node[rectangle] at (0, 3.8) {};
    \node[rectangle] at (0, 3.5) {\large $\C'$};
    \node[rectangle] at (2.75,0) {$L'$};
\end{tikzpicture}
\caption{Finding another place to put $L$ in \emph{Case~1} of Lemma~\ref{lem:main}.}
\label{fig:mainLemma2}
\end{figure}

Next we use $4$-edge-connectivity to ``find another place to put $L$''; this is the only time we will use $4$-edge-connectivity. Refer to Figure~\ref{fig:mainLemma2} for the following definitions. For this paragraph, consider only the transitions of the circuits in $\left(\C-\{(T_1, L, T_2)\}\right)\cup \{(T_1,T_2)\}$. Let $R$ be the transition which specifies the split $(T_1, T_2)$. Note that $R$ contains $h$ since, as we argued earlier, $h$ is in an arc of $T_1$ or $T_2$. Since $(G, \gamma, b)$ is $4$-edge-connected, there exists another transition $R' \neq R$ which is at a vertex $u$ that is hit by $L$. (It is possible that $u=v$.) Let $L'$ be a circuit which begins and ends at $u$ and is obtained by cyclically re-ordering $L$. Let $(S_1,S_2)$ be the split specified by $R'$. (It is possible that $(S_1, S_2)=(T_1, T_2)$.) By ``attaching $L'$ onto $(S_1, S_2)$'', we obtain a flooding $\C'$ so that $R$ is a transition of $\C'$ and $(S_1, L', S_2) \in \C'$.

Notice that when we removed $L$, we gained the non-zero circuit $(T_1,T_2)$. So by the optimality of $\C$, we must have lost a non-zero circuit when we added $L'$ to $(S_1, S_2)$. It follows that $\C'$ is optimal, $(S_1, L', S_2)$ is a zero circuit, and $B$ does not contain an element whose arc is in $(S_1, S_2)$. Thus, since $L$ is a non-zero circuit that contains an arc of an element in $B$, the set $B$ must contain a representative for either $(S_1, L', S_2)$ or $(S_2^{-1}, L', S_1^{-1})$. This completes the first case; we just replace $(S_1, L', S_2)$ by $(S_2^{-1}, L', S_1^{-1})$ in $\C'$ if necessary.

\smallskip
\noindent\emph{Case 2:} $\{h,h'\}$ and $\{r,r'\}$ are transitions of different circuits in~$\C$.
\smallskip

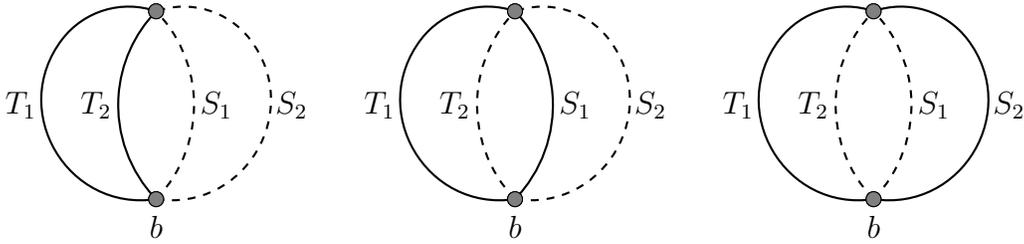
\begin{figure}
\centering
\begin{tikzpicture}[use Hobby shortcut, scale=1, every node/.style={MyNode}]
    \def \h {2.5}
    %\node[rectangle, draw=none, fill=none] at (0, \h+.6) {$\C$};
    \node[label=below:$b$] (a) at (0,0) {};
    \node (v) at (0,\h) {};
    \node[draw=none, fill=none] (A1) at (-1.25, .85*\h) {};
    \node[draw=none, fill=none] (A2) at (-.5, .55*\h) {};
    \node[draw=none, fill=none] (A3) at (.5, .55*\h) {};
    \node[draw=none, fill=none] (A4) at (1.25, .85*\h) {};
    \node[draw=none, fill=none] at (-1.8, .5*\h) {$T_1$};
    \node[draw=none, fill=none] at (-.8, .5*\h) {$T_2$};
    \node[draw=none, fill=none] at (.8, .5*\h) {$S_1$};
    \node[draw=none, fill=none] at (1.8, .5*\h) {$S_2$};
    \draw[thick] (a.center) .. (A1.center) .. (v.center);
    \draw[thick] (a.center) .. (A2.center) .. (v.center);
    \draw[thick, dashed] (a.center) .. (A3.center) .. (v.center);
    \draw[thick, dashed] (a.center) .. (A4.center) .. (v.center);
    \node at (a) {};
    \node at (v) {};
\end{tikzpicture}\hskip .43cm
\begin{tikzpicture}[use Hobby shortcut, scale=1, every node/.style={MyNode}]
    \def \h {2.5}
    %\node[rectangle, draw=none, fill=none] at (0, \h+.6) {$\C_1$};
    \node[label=below:$b$] (a) at (0,0) {};
    \node (v) at (0,\h) {};
    \node[draw=none, fill=none] (A1) at (-1.25, .85*\h) {};
    \node[draw=none, fill=none] (A2) at (-.5, .55*\h) {};
    \node[draw=none, fill=none] (A3) at (.5, .55*\h) {};
    \node[draw=none, fill=none] (A4) at (1.25, .85*\h) {};
    \node[draw=none, fill=none] at (-1.8, .5*\h) {$T_1$};
    \node[draw=none, fill=none] at (-.8, .5*\h) {$T_2$};
    \node[draw=none, fill=none] at (.8, .5*\h) {$S_1$};
    \node[draw=none, fill=none] at (1.8, .5*\h) {$S_2$};
    \draw[thick] (a.center) .. (A1.center) .. (v.center);
    \draw[thick, dashed] (a.center) .. (A2.center) .. (v.center);
    \draw[thick] (a.center) .. (A3.center) .. (v.center);
    \draw[thick, dashed] (a.center) .. (A4.center) .. (v.center);
    \node at (a) {};
    \node at (v) {};
\end{tikzpicture}\hskip .43cm
\begin{tikzpicture}[use Hobby shortcut, scale=1, every node/.style={MyNode}]
    \def \h {2.5}
    %\node[rectangle, draw=none, fill=none] at (0, \h+.6) {$\C_2$};
    \node[label=below:$b$] (a) at (0,0) {};
    \node (v) at (0,\h) {};
    \node[draw=none, fill=none] (A1) at (-1.25, .85*\h) {};
    \node[draw=none, fill=none] (A2) at (-.5, .55*\h) {};
    \node[draw=none, fill=none] (A3) at (.5, .55*\h) {};
    \node[draw=none, fill=none] (A4) at (1.25, .85*\h) {};
    \node[draw=none, fill=none] at (-1.8, .5*\h) {$T_1$};
    \node[draw=none, fill=none] at (-.8, .5*\h) {$T_2$};
    \node[draw=none, fill=none] at (.8, .5*\h) {$S_1$};
    \node[draw=none, fill=none] at (1.8, .5*\h) {$S_2$};
    \draw[thick] (a.center) .. (A1.center) .. (v.center);
    \draw[thick, dashed] (a.center) .. (A2.center) .. (v.center);
    \draw[thick, dashed] (a.center) .. (A3.center) .. (v.center);
    \draw[thick] (a.center) .. (A4.center) .. (v.center);
    \node at (a) {};
    \node at (v) {};
\end{tikzpicture}
\caption{The three ways to partition $\{T_1, T_2, S_1, S_2\}$ into two parts of size two, from \emph{Case~2} of Lemma~\ref{lem:main}.}
\label{fig:mainLemma3}
\end{figure}

Let $(T_1, T_2)$ and $(S_1,S_2)$ be the splits of $\C$ specified by $\{h,h'\}$ and $\{r,r'\}$, respectively. There are three ways to partition $\{T_1, T_2, S_1, S_2\}$ into two parts of size two, as depicted in Figure~\ref{fig:mainLemma3}. Each of these three ways yields a unique flooding of $(G, \gamma, b)$, up to reversing the two circuits that contain any of $h,h', r, r'$. We call these two circuits the \emph{new circuits} of the flooding. 

One of these three floodings is our original flooding $\C$; we are interested in the other two floodings, which we denote by $\C_1$ and $\C_2$.  Note that reversing the new circuits of $\C_1$ (respectively $\C_2$) does not affect whether or not $\C_1$ (respectively $\C_2$) is optimal. However, it might affect whether or not $B$ is a system of representatives. We make the choice as follows; if $L$ is a new circuit of $\C_1$ (respectively $\C_2$) so that more elements of $B$ have arcs in $L^{-1}$ than in $L$, then replace $L$ with $L^{-1}$. We claim that, with this choice, there exists $i \in \{1,2\}$ so that $\C_i$ is an optimal flooding and $B$ is a system of representatives for $\C_i$. This will complete the proof of Lemma~\ref{lem:main}. We now break into cases.

\smallskip
\noindent\emph{Case 2.1:} Both $(T_1, T_2)$ and $(S_1, S_2)$ are non-zero circuits.
\smallskip

Then, as a multi-set, $\{\gamma(T_1), \gamma(T_2), \gamma(S_1), \gamma(S_2)\}=\{0,0,1,1\}$. So for some $i \in \{1,2\}$, both of the new circuits of $\C_i$ are non-zero. Then $\C_i$ is an optimal flooding and $B$ is a system of representatives for $\C_i$.

\smallskip
\noindent\emph{Case 2.2:} Exactly one of $(T_1, T_2)$, $(S_1, S_2)$ is a non-zero circuit.
\smallskip

Then, as a multi-set, $\{\gamma(T_1), \gamma(T_2), \gamma(S_1), \gamma(S_2)\}$ is either $\{0,0,0,1\}$ or $\{0,1,1,1\}$. Note that there is only one way to partition $\{0,0,0,1\}$ into two parts of size two, and likewise for $\{0,1,1,1\}$. Thus both $\C_1$ and $\C_2$ are optimal. It just remains to consider whether $B$ is a system of representatives for $\C_1$ or $\C_2$.

Since exactly one of the new circuits has weight zero, there is only one representative to consider. Let $f$ be the arc of the element of $B$ that represents whichever of $(T_1, T_2)$, $(S_1, S_2)$ is a zero circuit. Note that whichever of $T_1, T_2, S_1, S_2$ has different parity from the other three must contribute to the non-zero circuit in all three of the floodings $\C$, $\C_1$, and $\C_2$. So $f$ is an arc of one of the other trails (that is, one of $T_1, T_2, S_1, S_2$ which has the same weight as two of the other trails). There are two ways to pair this trail containing $f$ with another trail of the same weight. It follows that $f$ is in a zero circuit in either $\C_1$ or $\C_2$, and $B$ is a system of representatives for that flooding.

\smallskip
\noindent\emph{Case 2.3:} Both $(T_1, T_2)$ and $(S_1, S_2)$ are zero circuits.
\smallskip

As $\C$ is optimal, it follows that, as a multi-set, $\{\gamma(T_1), \gamma(T_2), \gamma(S_1), \gamma(S_2)\}$ is either $\{0,0,0,0\}$ or $\{1,1,1,1\}$. So both $\C_1$ and $\C_2$ are optimal. Now let $f_1$ (respectively $f_2$) be the arc of the element in $B$ that represents $(T_1, T_2)$ (respectively $(S_1, S_2)$). Then $f_1$ and $f_2$ are in distinct circuits in either $\C_1$ or $\C_2$, and $B$ is a system of representatives for that flooding. 

This completes all possible cases and therefore the proof of Lemma~\ref{lem:main}.
\end{proof}

Now we are ready to prove that the basis exchange axiom holds, which is the final lemma of this section. As the flooding matroid always has a basis (possibly the empty set), this lemma proves that the flooding matroid is in fact a matroid.

\begin{lemma}
\label{lem:basisExchange}
For any \rooted{} $(G, \gamma, b)$, bases $B_1$ and $B_2$ of $M(G, \gamma, b)$, and $b_1 \in B_1-B_2$, there exists $b_2 \in B_2-B_1$ so that $(B_1-\{b_1\}) \cup \{b_2\}$ is a basis of $M(G, \gamma, b)$.
\end{lemma}
\begin{proof}
Going for a contradiction, suppose that the lemma is false. Then choose a counterexample so that $(G, \gamma, b)$ has as few edges as possible, and, subject to that, as many vertices as possible. This assumption may seem strange now but will prove to be convenient later. Such a choice is possible since an Eulerian graph with $m$ edges has at most $m$ vertices. 

Our aim is to apply Lemma~\ref{lem:main}. So we need a vertex other than $b$, and we need $(G, \gamma, b)$ to be $4$-edge-connected. We take care of these things now.

\begin{claim}
There exists a vertex other than $b$.
\end{claim}
\begin{proof}
If not, then every zero circuit in a flooding consists of a single loop $f$ and must be represented by $(f, 0)$. We can reverse such a circuit to obtain a zero circuit represented by $(f^{-1},0)$. Then the element $b_1$ is of the form $(f,0)$, and we can take $b_2$ to be the element $(f^{-1},0) \in B_2-B_1$.
\end{proof}

The next claim is actually the hardest part of the proof.

\begin{claim}
\label{clm:4edgeConn}
The graph $(G, \gamma, b)$ is $4$-edge-connected.
\end{claim}
\begin{proof}
Otherwise, there exists a set $Y \subseteq V(G)-\{b\}$ with $|\delta(Y)|=2$. Let $(\hat{G}, \hat{\gamma}, b)$ be the \rooted{} that is obtained from $(G, \gamma, b)$ by deleting all vertices in $Y$ and then adding a new edge $\hat{e}$ whose ends are the neighbours of $Y$ (possibly $\hat{e}$ is a loop) and whose weight is the sum of the weights of the edges in $E(Y) \cup \delta(Y)$. Note that $\tilde{\nu}(\hat{G}, \hat{\gamma}, b)=\tilde{\nu}(G, \gamma, b)$. The proof of the claim is fairly straightforward from here; we apply Lemma~\ref{lem:basisExchange} to the graph $(\hat{G}, \hat{\gamma}, b)$, which has fewer edges than $(G, \gamma, b)$. It is somewhat technical to state this precisely though. We begin by giving some definitions related to $B_1$ and $B_2$. So let $i \in \{1,2\}$. 

Fix an optimal flooding $\C_i$ of $(G, \gamma, b)$ so that $B_i$ is a system of representatives for $\C_i$. Let $T_i$ be the unique subtrail of a circuit in $\C_i$ so that the edge-set of $T_i$ is $E(Y) \cup \delta(Y)$. Then there exists an optimal flooding $\hat{\C}_i$ of $(\hat{G}, \hat{\gamma}, b)$ which is obtained from $\C_i$ by replacing $T_i$ with an arc $\hat{f}_i$ whose edge is $\hat{e}$. Now, if no element of $B_i$ has an arc in $T_i$, then $B_i$ is also a system of representatives for $\hat{\C}_i$ and we set $\hat{B}_i \coloneqq B_i$. Otherwise, let $(f_i, \alpha_i) \in B_i$ be the element whose arc is in $T_i$; then there exists $\hat{\alpha}_i \in \{0,1\}$ so that $(B_i-\{(f_i, \alpha_i)\}) \cup \{(\hat{f}_i, \hat{\alpha}_i)\}$ is a system of representatives for $\hat{\C}_i$, and we let $\hat{B}_i$ be this set. This completes the definitions.

Next we apply Lemma~\ref{lem:basisExchange} to $(\hat{G}, \hat{\gamma}, b)$, which has fewer edges than $(G, \gamma, b)$. So let $\hat{b}_1$ be the element of $\hat{B}_1$ that corresponds to $b_1$. It is possible that $\hat{b}_1$ is in $\hat{B}_2$. In this case, $\hat{b}_1=(\hat{f}_1, \hat{\alpha}_1)=(\hat{f}_2, \hat{\alpha}_2)$, and $(B_1-\{b_1\}) \cup \{(f_2, \alpha_2)\}$ is a basis of $M(G, \gamma, b)$. To see this, note that it is a system of representatives for the flooding that is obtained from $\hat{\C}_1$ by replacing the arc $\hat{f}_1$ with the trail $T_2$.

So we may assume that $\hat{b}_1 \in \hat{B}_1-\hat{B}_2$. Then there exists $\hat{b}_2 \in \hat{B}_2-\hat{B}_1$ so that $(\hat{B}_1-\{\hat{b}_1\}) \cup \{\hat{b}_2\}$ is a basis of $M(\hat{G}, \hat{\gamma}, b)$. If $\hat{b}_2$ is in $B_2$ as well, then $(B_1-\{b_1\}) \cup \{\hat{b}_2\}$ is a basis of $M(G, \gamma, b)$; we replace $\hat{f}_1$ or its reversal by $T_1$ or its reversal. Otherwise, $\hat{b}_2=(\hat{f}_2, \hat{\alpha}_2)$, and instead $(B_1-\{b_1\}) \cup \{(f_2, \alpha_2)\}$ is a basis of $M(G, \gamma, b)$; we replace $\hat{f}_2$ by $T_2$. We note that $(f_2, \alpha_2)$ is not in $B_1$ simply because $(B_1-\{b_1\}) \cup \{(f_2, \alpha_2)\}$ corresponds to an optimal flooding and therefore has the same size as $B$. This completes the proof of Claim~\ref{clm:4edgeConn}.
\end{proof}

Now, fix a vertex $v \neq b$ which is joined to $b$ by an edge $e$; such things exist since $G$ is connected and has a vertex other than $b$.  Let $h$ be the half-edge of $e$ which is at $v$. By Lemma~\ref{lem:main} applied to $B_1$ and $B_2$, there exists a transition $\{h, h'\}$ at $v$ so that there are optimal floodings $\C_1$ and $\C_2$ which both have $\{h, h'\}$ as a transition, and which have $B_1$ and $B_2$ (respectively) as a system of representatives. Let $(\hat{G}, \hat{\gamma}, b)$ be the \rooted{} that is obtained from $(G, \gamma, b)$ by adding a new vertex $v'$ and making the half-edges $h$ and $h'$ incident to $v'$ instead of $v$. Then $B_1$ and $B_2$ are both bases of $M(\hat{G}, \hat{\gamma}, b)$. Moreover, Lemma~\ref{lem:basisExchange} holds for $(\hat{G}, \hat{\gamma}, b)$ since it has the same number of edges as $(G, \gamma, b)$ but more vertices. It follows that the lemma holds for $(G, \gamma, b)$ as well. This is a contradiction and completes the proof of Lemma~\ref{lem:basisExchange}.
\end{proof}

%The Reduction Step
\section{Completing the proof}
\label{sec:reduction}

As discussed earlier, we will prove Theorem~\ref{thm:mainFlooding} by reducing to the case that the flooding matroid has rank~$1$. The key step is to show that in a ``minimum counterexample'', every arc which is not incident to the root is in two non-loop elements of the flooding matroid. We  accomplish this step in the next two lemmas. First we need to define the ``reduction move''; refer to Figure~\ref{fig:reductionMove} for an example. 

\begin{figure}
\centering
\begin{tikzpicture}[scale=1, every node/.style={MyNode}]
    \def \w {3.75}
    \def \h {1.5}
    \node[label=below:$b$] (a) at (.5*\w,0) {};
    \node[rectangle, draw=none, fill=none] (G) at (.5*\w,-1.1) {$(G, \gamma, b)$};
    \node[rectangle, draw=none, fill=none, label=below:$e$] (G) at (.5*\w,\h) {};
    %Draw component
    \node (A1) at (0,\h) {};
    \node (A2) at (\w,\h) {};
    \node (A3) at (.5*\w,1.35*\h) {};
    \node (A4) at (.5*\w, 2*\h) {};
    \draw[thick] (A1) -- (A2) -- (A3) -- (A1) -- (A4) -- (A2);
    \draw[thick] (A3) to [bend left=14] (A4);
    \draw[thick] (A3) to [bend right=14] (A4);
    \draw[MyRedArc] (A1)-- (A2) -- (A4);
    \draw[MyRedArc] (A1) -- (a);
    \draw[thick] (A2) -- (a);
\end{tikzpicture}\hskip 1.5cm
\begin{tikzpicture}[scale=1, every node/.style={MyNode}]
    \def \h {1.5}
    \node[draw=none, fill=none] (A4) at (0,2*\h) {};
    \node[draw=none, fill=none] (bottom) at (0,-1.1) {};
    \node[rectangle, draw=none, fill=none] (center) at (0, \h) {\Large$\longrightarrow$};
\end{tikzpicture}\hskip 1.5cm
\begin{tikzpicture}[scale=1, every node/.style={MyNode}]
    \def \w {3.75}
    \def \h {1.5}
    \node[label=below:$b$] (a) at (.5*\w,0) {};
    \node[rectangle, draw=none, fill=none] (G) at (.5*\w,-1.1) {$(\hat{G}, \hat{\gamma}, b)$};
    %Draw component
    \node (A1) at (0,\h) {};
    \node (A2) at (\w,\h) {};
    \node (A3) at (.5*\w,1.35*\h) {};
    \node (A4) at (.5*\w, 2*\h) {};
    \draw[thick] (A2) -- (A3) -- (A1) -- (A4) -- (A2);
    \draw[thick] (A3) to [bend left=14] (A4);
    \draw[thick] (A3) to [bend right=14] (A4);
    \draw[MyRedArc] (A2) -- (A4);
    \draw[MyRedArc] (A1) -- (a);
    \draw[thick] (A2) -- (a);
    \draw[thick] (A1) to [bend left=24] (a);
    \draw[MyRedArc] (A2) to [bend right=24] (a);
    %transition
    \node[draw=none, fill=none] (a') at (.5*\w,.15) {};
    \draw[ultra thick, dashed] (.5*\w-.09*\w, .47*\h) -- (a');
    \draw[ultra thick, dashed] (.5*\w+.09*\w, .47*\h) -- (a');
    \node[rectangle, draw=none, fill=none] at (.5*\w,.3) {};
\end{tikzpicture}
\caption{An \rooted{} (left), and an $e$-reduction of it (right).}
\label{fig:reductionMove}
\end{figure}
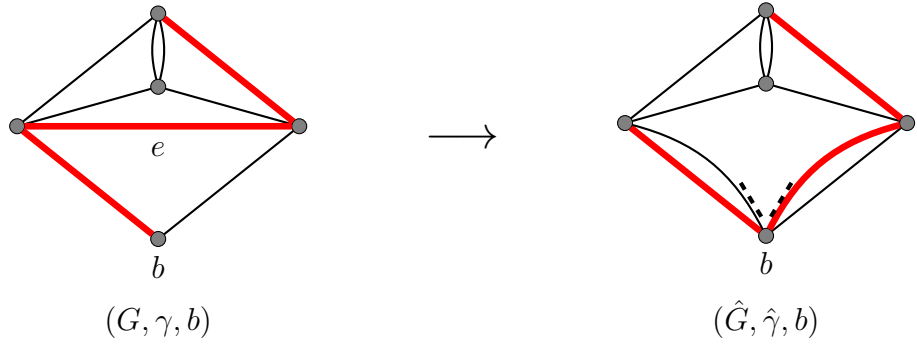

Let $(G, \gamma, b)$ be an \rooted{}, and let $e$ be an edge of $G-b$. For convenience, let $h$ and $r$ be the half-edges so that $e = \{h,r\}$. Then an \emph{$e$-reduction of $(G, \gamma, b)$} is any \rooted{} $(\hat{G}, \hat{\gamma}, b)$ that has a transition $\{h',r'\}$ at $b$ so that \begin{enumerate}
    \item $\hat{G}$ is obtained from $G$ by deleting $e$ and adding the edges $\{h,h'\}$ and $\{r,r'\}$, and
    \item the weights of $\{h,h'\}$ and $\{r,r'\}$ (according to $\hat{\gamma}$) sum to $\gamma(e)$.
\end{enumerate}
\noindent Note that $(\hat{G}, \hat{\gamma}, b)$ is not unique since there are two ways to weight its new edges. Furthermore, $\hat{G}-b$ has fewer edges than $G-b$; this is the sense in which $\hat{G}$ is ``smaller'' than $G$. Sometimes we do not want to specify the edge $e$; so we say that a \emph{reduction of $(G, \gamma, b)$} is any \rooted{} which is an $e$-reduction of $(G, \gamma, b)$ for some edge $e$ of $G-b$.

The first lemma shows that if there is a reduction of $(G, \gamma, b)$ whose flooding number is not too much larger, then we can ``lift'' its certificate to $(G, \gamma, b)$.

\begin{lemma}
\label{lem:critical}
If $(G, \gamma, b)$ is an \rooted{} which has a reduction $(\hat{G}, \hat{\gamma}, b)$ so that $(\hat{G}, \hat{\gamma}, b)$ has flooding number at most $\tilde{\nu}(G, \gamma, b)+1$ and $(\hat{G}, \hat{\gamma}, b)$ has a certificate, then $(G, \gamma, b)$ also has a certificate. 
\end{lemma}
\begin{proof}
Let $e$ be the edge of $G-b$ so that $(\hat{G}, \hat{\gamma}, b)$ is an $e$-reduction of $(G, \gamma, b)$. By assumption, $\tilde{\nu}(\hat{G}, \hat{\gamma}, b) \leq \tilde{\nu}(G, \gamma, b)+1$. Moreover, since $\gamma(E(G))$ has the same parity as $\hat{\gamma}(E(\hat{G}))$, the flooding numbers also have the same parity. Thus we actually have that $\tilde{\nu}(\hat{G}, \hat{\gamma}, b)\leq \tilde{\nu}(G, \gamma, b)$. Also, since any flooding of $(G, \gamma, b)$ with $k$ non-zero circuits yields a flooding of $(\hat{G}, \hat{\gamma}, b)$ with at least $k$ non-zero circuits, we have that $\tilde{\nu}(G, \gamma, b)\leq \tilde{\nu}(\hat{G}, \hat{\gamma}, b)$. So in fact $\tilde{\nu}(G, \gamma, b)= \tilde{\nu}(\hat{G}, \hat{\gamma}, b)$.

Let $(X, \hat{\gamma}')$ be a certificate of $(\hat{G}, \hat{\gamma}, b)$. By performing the same sequence of shiftings in $(G, \gamma, b)$, we can find a shifting $\gamma'$ of $\gamma$ so that $(\hat{G}, \hat{\gamma}', b)$ is an $e$-reduction of $(G, \gamma', b)$. We now show that $(X, \gamma')$ is a certificate for $(G, \gamma, b)$ by breaking into cases based on where the ends of $e$ ``lie''.

\smallskip
\noindent\emph{Case 1:} Both ends of $e$ are in $X$.
\smallskip

If $\gamma'(e) = 0$ then certainly $(X, \gamma')$ is a certificate. Otherwise, one of the new edges of $\hat{G}$ is non-zero according to $\hat{\gamma}'$, and again $(X, \gamma')$ is a certificate.

\smallskip
\noindent\emph{Case 2:} Exactly one end of $e$ is in $X$.
\smallskip

Then the other end of $e$ is in a component of $\hat{G}-X$; write $Y$ for the vertex-set of that component. The only way $(X, \gamma')$ might not be a certificate is if $Y$ is odd in $(\hat{G}, \hat{\gamma}', b)$ but not in $(G, \gamma', b)$. However, if this occurs, then the new edge of $(\hat{G}, \hat{\gamma}', b)$ which has both ends in $X$ is non-zero. Therefore, it contributed to $\hat{\gamma}'(E(X))$ but not to $\gamma'(E(X))$; so again $(X, \gamma')$ is a certificate.

\smallskip
\noindent\emph{Case 3:} The ends of $e$ are in the same component of $\hat{G}-X$.
\smallskip

Then $|\delta_{G}(X)|=|\delta_{\hat{G}}(X)|-2$, and regardless of whether the vertex set of that component is odd in $(G, \gamma', b)$, we still have that $(X, \gamma')$ is a certificate.

\begin{figure}
\centering
\begin{tikzpicture}[scale=1, every node/.style={MyNode}]
    \def \w {2}
    \def \boxX {-.6}
    \def \boxY {-.6}
    \def \boxH {1.8}
    \def \boxW {5}
    %Draw A
    \draw[dashed] (\boxX, \boxY) --++ (\boxW,0) --++ (0,-\boxH) --++ (-\boxW,0) -- cycle;
    \node[label=center:$X$, draw=none, fill=none] (A) at (\boxX+.3,\boxY-\boxH+.3) {};
    \node[label=below:$b$] (a) at (\boxX+.68*\boxW,\boxY-\boxH+.6) {};
    \node[rectangle, draw=none, fill=none] (G) at (\boxX+.5*\boxW,\boxY-1*\boxH-.5) {$(\hat{G}, \hat{\gamma}', b)$};
    %Draw component to left
    \node (A1) at (0,0) {};
    \node (A2) at (\w,0) {};
    \node (A3) at (.5*\w,.22*\w) {};
    \node (A4) at (.5*\w,.66*\w) {};
    \draw[thick] (A1) -- (A2) -- (A3) -- (A1) -- (A4) -- (A2);
    \draw[thick] (A3) to [bend left=14] (A4);
    \draw[thick] (A3) to [bend right=14] (A4);
    \draw[MyRedArc] (A2) -- (A4);
    \node (A1') at (-.15, \boxY -.4) {};
    \draw[MyRedArc] (A1) -- (A1');
    \draw[thick] (A2) -- (a);
    %Draw component to right
    \node (B1) at (\boxX+\boxW-.8,0) {};
    \node (B1') at (\boxX+\boxW-.8+.35, \boxY -.4) {};
    \draw[thick] (B1) -- (B1');
    \draw[MyRedArc] (B1) -- (a);
    \draw[MyRedArc] (B1) to [out=65,in=115,looseness=30] (B1);
\end{tikzpicture}\hskip 1.2cm
\begin{tikzpicture}[scale=1, every node/.style={MyNode}]
    \def \w {2}
    \def \boxX {-.6}
    \def \boxY {-.6}
    \def \boxH {1.8}
    \def \boxW {5}
    \node[draw=none, fill=none] (A4) at (0,.66*\w) {};
    \node[draw=none, fill=none] (bottom) at (0,\boxY-\boxH) {};
    \node[rectangle, draw=none, fill=none] (center) at (0, .33*\w+.5*\boxY-.5*\boxH) {\Large$\longrightarrow$};
\end{tikzpicture}\hskip 1.2cm
\begin{tikzpicture}[scale=1, every node/.style={MyNode}]
    \def \w {2}
    \def \boxX {-.6}
    \def \boxY {-.6}
    \def \boxH {1.8}
    \def \boxW {5}
    %Draw A
    \draw[dashed] (\boxX, \boxY) --++ (\boxW,0) --++ (0,-\boxH) --++ (-\boxW,0) -- cycle;
    \node[label=center:$X$, draw=none, fill=none] (A) at (\boxX+.3,\boxY-\boxH+.3) {};
    \node[label=below:$b$] (a) at (\boxX+.68*\boxW,\boxY-\boxH+.6) {};
    \node[rectangle, draw=none, fill=none] (G) at (\boxX+.5*\boxW,\boxY-1*\boxH-.5) {$(G, \gamma', b)$};
    %Draw component to left
    \node (A1) at (0,0) {};
    \node (A2) at (\w,0) {};
    \node (A3) at (.5*\w,.22*\w) {};
    \node (A4) at (.5*\w,.66*\w) {};
    \draw[thick] (A1) -- (A2) -- (A3) -- (A1) -- (A4) -- (A2);
    \draw[thick] (A3) to [bend left=14] (A4);
    \draw[thick] (A3) to [bend right=14] (A4);
    \draw[MyRedArc] (A2) -- (A4);
    \node (A1') at (-.15, \boxY -.4) {};
    \draw[MyRedArc] (A1) -- (A1');
    %\draw[thick] (A2) -- (a);
    %Draw component to right
    \node(B1) at (\boxX+\boxW-.8,0) {};
    \node (B1') at (\boxX+\boxW-.8+.35, \boxY -.4) {};
    \draw[thick] (B1) -- (B1');
    %\draw[thick] (B1) -- (a);
    \draw[MyRedArc] (A2) -- (B1) node[midway,above, draw=none, fill=none]{\textcolor{black}{$e$}};
    \draw[MyRedArc] (B1) to [out=65,in=115,looseness=30] (B1);
\end{tikzpicture}
\caption{A depiction for \emph{Case 4} of Lemma~\ref{lem:critical}, where $Y_1$ and $Y_2$ are both odd in $(\hat{G}, \hat{\gamma}', b)$.}
\label{fig:critical}
\end{figure}
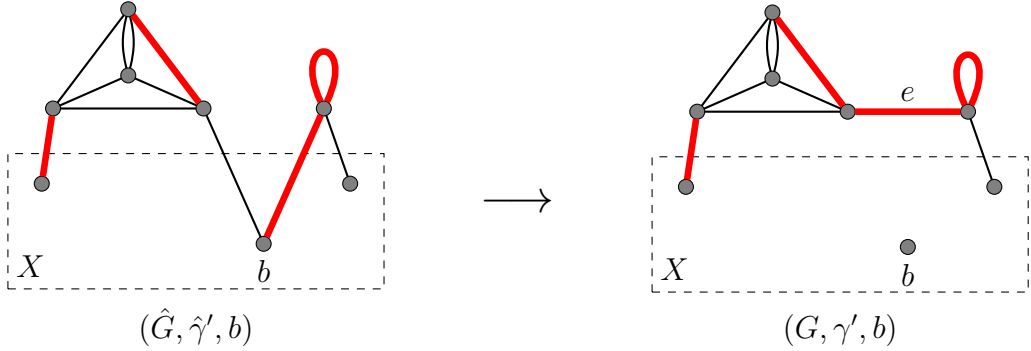

\smallskip
\noindent\emph{Case 4:} The ends of $e$ are in different components of $\hat{G}-X$.
\smallskip

Let $Y_1$ and $Y_2$ be the vertex-sets of those two components. Again we have that $|\delta_{G}(X)|=|\delta_{\hat{G}}(X)|-2$. So the only possible problem is if $Y_1$ and $Y_2$ are both odd in $(\hat{G}, \hat{\gamma}', b)$ and $Y_1 \cup Y_2$ is not odd in $(G, \gamma', b)$. However this cannot occur since $|\delta_G(Y_1 \cup Y_2)|/2=|\delta_{\hat{G}}(Y_1)|/2+|\delta_{\hat{G}}(Y_2)|/2-1$, while the parity of the number of relevant non-zero edges just sums (as in Figure~\ref{fig:critical}). So indeed $(X, \gamma')$ is a certificate for $(G, \gamma, b)$.

This completes all of the cases and therefore also the proof of Lemma~\ref{lem:critical}.
\end{proof}

The second lemma essentially says that if Lemma~\ref{lem:critical} cannot be applied, then the flooding matroid has many non-loop elements.

\begin{lemma}
\label{lem:loops}
If $(G, \gamma, b)$ is an \rooted{} whose reductions all have flooding number at least $\tilde{\nu}(G, \gamma, b)+2$, then each arc of $G-b$ is in two non-loop elements of the flooding matroid $M(G, \gamma, b)$.
\end{lemma}
\begin{proof}
Let $f$ be an arc of $G-b$, and let $e$ be the corresponding edge. We are trying to prove that for $i=0,1$, there exists an optimal flooding of $(G, \gamma, b)$ that contains a zero circuit represented by $(f, i)$. 

Let $(\hat{G}, \hat{\gamma}_0, b)$ be an $e$-reduction of $(G, \gamma, b)$; such an \rooted{} exists. Let $\hat{e}_1$ and $\hat{e}_2$ denote its two new edges. By adding~$1$ to the weights of $\hat{e}_1$ and $\hat{e}_2$, we can obtain another $e$-reduction $(\hat{G}, \hat{\gamma}_1, b)$ of $(G, \gamma, b)$. Both of these two new \rooteds{} have flooding number at least $\tilde{\nu}(G, \gamma, b)+2$. Let $\hat{\C}_0$ and $\hat{\C}_1$ be optimal floodings of $(\hat{G}, \hat{\gamma}_0, b)$ and $(\hat{G}, \hat{\gamma}_1, b)$, respectively. First we prove a claim.

\begin{claim}
\label{clm:nonZeroCirc}
Neither $\hat{e}_1$ nor $\hat{e}_2$ is in a zero circuit in $\hat{\C}_0$ or $\hat{\C}_1$.
\end{claim}
\begin{proof}
Going for a contradiction, suppose that for some $i\in \{0,1\}$, the flooding $\hat{C}_i$ of $(\hat{G}, \hat{\gamma}_i, b)$ does contain such a zero circuit. If $\hat{e}_1$ and $\hat{e}_2$ are in different circuits in $\hat{C}_i$, then we obtain a contradiction to the fact that the reduction has larger flooding number. Otherwise, $\hat{e}_1$ and $\hat{e}_2$ are in the same circuit $\hat{C} \in \hat{C}_i$. Then in $(G, \gamma, b)$, the circuit $\hat{C}$ becomes a circuit which contains $e$. Since $(G, \gamma, b)$ is connected, this circuit can be ``attached'' back onto some other circuit in $\hat{\C}_i-\{\hat{C}\}$. This again contradicts the fact that the reduction has larger flooding number.
\end{proof}

Now we break into two cases based where $\hat{e}_1$ and $\hat{e}_2$ ``lie'' in $\hat{\C}_0$ and $\hat{\C}_1$. 

\smallskip
\noindent\emph{Case~1:} There exists $i \in \{0,1\}$ so that $\hat{e}_1$ and $\hat{e}_2$ are in the same circuit in $\hat{\C}_i$.
\smallskip

Let $\hat{C} \in \hat{\C}_i$ be that circuit. Then $\hat{C}$ is non-zero by Claim~\ref{clm:nonZeroCirc}. So, similarly to before, we can obtain a flooding $\C$ of $(G, \gamma, b)$ by ``attaching'' this circuit onto some circuit of $\hat{\C}_i-\{\hat{C}\}$. Since the flooding number of the reduction is least two higher, $\C$ is an optimal flooding of $(G, \gamma, b)$. Moreover, after possibly reversing a circuit in $\C$, we can find trails $T_1, C, T_2$ so that $(T_1, C, T_2)$ is a zero circuit in $\C$, and $C$ is a non-zero circuit which contains the arc $f$. We can obtain another optimal flooding of $(G, \gamma, b)$ by replacing $(T_1, C, T_2)$ with the circuit $(T_2^{-1}, C, T_1^{-1})$. Then one of $(T_1, C, T_2)$, $(T_2^{-1}, C, T_1^{-1})$ is represented by $(f,0)$, the other is represented by $(f,1)$. That completes this case.

\smallskip
\noindent\emph{Case~2:} The edges $\hat{e}_1$ and $\hat{e}_2$ are in distinct circuits in each of $\hat{\C}_0$ and $\hat{\C}_1$.
\smallskip

Again we use the fact that these circuits are non-zero by Claim~\ref{clm:nonZeroCirc}. Now, up to symmetry between $\hat{e}_1$ and $\hat{e}_2$, we may assume that $f$ is an ordered pair of half-edges $(h_1, h_2)$, where $h_1$ is a half-edge of $\hat{e}_1$ and $h_2$ is a half-edge of $\hat{e}_2$. Then the flooding $\hat{\C}_0$ of $(\hat{G}, \hat{\gamma}_0, b)$ yields an optimal flooding of $(G, \gamma, b)$ where $(f, 1+\hat{\gamma}_0(\hat{e}_2))$ represents a zero circuit. Likewise, $\hat{\C}_1$ yields an optimal flooding of $(G, \gamma, b)$ where $(f, 1+\hat{\gamma}_1(\hat{e}_2))$ represents a zero circuit. Since $\hat{\gamma}_0(\hat{e}_2) \neq \hat{\gamma}_1(\hat{e}_2)$, this completes the proof of Lemma~\ref{lem:loops}.
\end{proof}

We are ready to prove Theorem~\ref{thm:mainFlooding}, which is restated below for convenience.

\thmMain*
%\addtocounter{theorem}{-7}
\begin{proof}
Suppose for a contradiction that the theorem is false. Let $(G, \gamma, b)$ be a counterexample so that $|E(G-b)|$ is as small as possible and, subject to that, so that $|E(G)|$ is as small as possible. Recall that we have already shown one direction of the inequality in Lemma~\ref{lem:easyDirection}. 

First we prove a couple of claims, and then we use them to show that the flooding matroid $M(G, \gamma, b)$ has rank~$1$. The theorem will follow shortly after.

\begin{claim}
\label{clm:twoNonLoop}
Each arc of $G-b$ is in two non-loop elements of $M(G, \gamma, b)$.
\end{claim}
\begin{proof}
By the choice of $(G, \gamma, b)$, all of its reductions have a certificate. So by Lemma~\ref{lem:critical}, all reductions have flooding number at least $\tilde{\nu}(G, \gamma, b)+2$. Thus, by Lemma~\ref{lem:loops}, each arc of $G-b$ is in two non-loop elements of $M(G, \gamma, b)$.
\end{proof}

\begin{claim}
\label{clm:basic}
There are no loops at $b$, the graph $G-b$ is connected, and $E(G-b)$ is non-empty.
\end{claim}
\begin{proof}
There is no loop at $b$ since, otherwise, any certificate for the graph obtained from $(G, \gamma, b)$ by deleting that loop also yields a certificate for $(G, \gamma, b)$. 

Now suppose for a contradiction that the graph $G-b$ is not connected. Then there are \rooteds{} $(G_1, \gamma_1, b)$ and $(G_2, \gamma_2, b)$ so that $G=G_1 \cup G_2$, the only vertex in common between $G_1$ and $G_2$ is $b$, and both $V(G_1)-\{b\}$ and $V(G_2)-\{b\}$ are non-empty. Then by the choice of $(G, \gamma, b)$, it follows that there are certificates $(X_1, \gamma_1')$ and $(X_2, \gamma_2')$ for $(G_1, \gamma_1, b)$ and $(G_2, \gamma_2, b)$, respectively. We may assume that $\gamma_1'$ and $\gamma_2'$ are obtained without shifting at $b$; any time we wish to shift at $b$, we can instead shift at every vertex other than $b$. Thus there exists a shifting $\gamma'$ of $\gamma$ which agrees with $\gamma_1'$ on $E(G_1)$ and $\gamma_2'$ on $E(G_2)$. Then $\tilde{\nu}(G, \gamma, b)=\tilde{\nu}(G_1, \gamma_1, b)+\tilde{\nu}(G_2, \gamma_2, b)$, and $(X_1 \cup X_2, \gamma')$ is a certificate for $(G, \gamma, b)$. This is a contradiction, which shows that $G-b$ is connected.

Finally, suppose for a contradiction that $E(G-b)$ is empty. From the last two paragraphs, this means that $(G, \gamma, b)$ has two vertices and no loops. Let $\C$ be an optimal flooding of $(G, \gamma, b)$. If $\C$ has no zero circuits, then $\tilde{\nu}(G, \gamma, b)=\deg(b)/2$ and $(\{b\}, \gamma)$ is a certificate. So we may assume that $\C$ contains a zero circuit $C$. Then, after possibly shifting at $b$, we may assume that both edges of $C$ have weight zero. Then, since every other circuit of $\C$ hits $C$ at the vertex other than $b$, this means that every zero circuit in $\C$ has both of its edges of weight zero (otherwise $\C$ would not be optimal). It follows that $(V(G), \gamma)$ is a certificate. This is again a contradiction and completes the proof of Claim~\ref{clm:basic}.
\end{proof}

The next claim almost completes the proof of the theorem.

\begin{claim}
\label{clm:rank}
The matroid $M(G, \gamma, b)$ has rank~$1$.
\end{claim}
\begin{proof}
Let $F$ be the set of all elements of $M(G, \gamma, b)$ whose arc is not incident to $b$. We first prove that $F$ is contained in a {parallel class} of $M(G, \gamma, b)$; that is, we prove that $F$ has no loops and has rank exactly~$1$ in $M(G, \gamma, b)$. 

By Claim~\ref{clm:basic}, $E(G-b)$ is non-empty. By Claim~\ref{clm:twoNonLoop}, each arc of $G-b$ is in two non-loop elements of the flooding matroid. Thus $F$ has rank at least~$1$, and $F$ has no elements which are loops in the matroid. Furthermore, for each arc $f$ of $G-b$, the four elements in $F$ whose arc is $f$ or $f^{-1}$ are all parallel in the matroid. By Lemma~\ref{lem:transitivity}, for any arcs $f_0$ and $f_1$ of $G-b$ with the same head, the elements $(f_0, 0)$ and $(f_0, 1)$ in $F$ are parallel. So by the transitivity of parallel pairs, for any vertex $v \neq b$, all of the elements of $F$ whose arcs are incident to $v$ are parallel. Thus, using the transitivity of parallel pairs and the fact that $G-b$ is connected (see Claim~\ref{clm:basic}), $F$ is contained in a parallel class of $M(G, \gamma, b)$.

Now, consider an arc $f$ whose tail is $b$. By Claim~\ref{clm:basic}, there exists an arc $f'$ of $G-b$ which has the same head as $f$. By Lemma~\ref{lem:transitivity}, there is no basis which contains both $(f,0)$ and $(f', 1)$, and there is no basis which contains both $(f,1)$ and $(f', 0)$. Also, note that $(f^{-1},1)$ is a loop element of the matroid since $f^{-1}$ has $b$ as its head. Finally, notice that $(f^{-1},0)$ is a loop element of the matroid if and only if $(f,\gamma(f))$ is a loop element of the matroid. It follows from the transitivity of parallel pairs and the fact that $F$ is contained in a parallel class that all non-loop elements of $M(G, \gamma, b)$ are in the same parallel class. So $M(G, \gamma, b)$ has rank~$1$, as desired.
\end{proof}

Since $M(G, \gamma, b)$ has rank~$1$, the flooding number of $(G,\gamma, b)$ is $\deg(b)/2-1$. So the parity of $\gamma(E(G))$ is different from the parity of $\textrm{deg}(b)/2$. Since there are no loops at $b$ and $G-b$ is connected by Claim~\ref{clm:basic}, we get that $(\{b\}, \gamma)$ is a certificate with one odd component. This is a contradiction, which completes the proof of Theorem~\ref{thm:mainFlooding}.
\end{proof}
%\addtocounter{theorem}{+7}

\section{Corollaries}
\label{sec:cor}

In this section we prove the corollaries of Theorem~\ref{thm:mainFlooding} which were mentioned in the introduction, as well two more corollaries of interest.

\subsection*{Regular graphs}
We begin by proving the two conjectures of M\'{a}\v{c}ajov\'{a} and \v{S}koviera~\cite{MacajovaOddEul} about regular graphs. 

First we prove a corollary about ``rooted'' graphs which are $2\ell$-regular except for possibly the root, whose degree can be smaller. We consider the signature where every edge has weight~$1$. The corollary says that if the root has degree $2d$ and the graph is $2d$-edge-connected and has an odd number of vertices, then the flooding number is~$d$. Note that we cannot replace the condition that ``there are an odd number of vertices'' with the condition that ``$|E(G)|$ and $d$ have the same parity''; in that case the graph could be bipartite (for instance when $d=\ell$ and $\ell$ is even) and thus have flooding number zero.

When $d = \ell$ the following corollary is precisely Conjecture~2 from~\cite{MacajovaOddEul}. After proving the corollary, we use it to prove Corollary~\ref{cor:corMS} from the introduction.

\begin{corollary}
\label{cor:almostReg}
For any positive integers $\ell$ and $d$ with $d \leq \ell$ and any $2d$-edge-connected graph $G$ with an odd number of vertices and a vertex $b$ of degree~$2d$ so that every other vertex has degree $2\ell$, there exists a circuit-decomposition of size $d$ where each circuit has an odd number of edges and begins and ends at $b$.
\end{corollary}
\begin{proof}
Throughout the proof we write $n \equiv m$ to mean that integers $n$ and $m$ are equivalent modulo~$2$. The proof is straightforward, although the case analysis is somewhat tedious. First of all, we may assume that $d \geq 2$ since otherwise the corollary holds just because $G$ is an Eulerian graph with an an odd number of edges. (Note that $|E(G)|\equiv \ell(|V(G)|-1)+d \equiv d$ since $|V(G)|$ is odd.) 

Now let $\gamma$ denote the signature of $G$ where every edge is given weight~$1$; thus we are trying to show that the \rooted{} $(G, \gamma, b)$ has flooding number $d$. By Theorem~\ref{thm:mainFlooding}, this \rooted{} has a certificate $(X, \gamma')$. That is, $X$ is a set of vertices which contains $b$ and $\gamma'$ is a shifting of $\gamma$ so that the flooding number is equal to\begin{align}
    \label{eqn:cor}
    \gamma'(E(X))+\frac{1}{2}|\delta(X)|-{\mathrm{odd}}_{\gamma'}(G-X).
\end{align}

Shifting at a vertex outside of $X$ does not change equation~(\ref{eqn:cor}); in particular, it does not change the odd components because every vertex has even degree. (If we shift at a vertex inside a set $Y\subseteq V(G)$, then the parity of $\gamma'(E(Y) \cup \delta(Y))$ does not change.) Thus we may assume that $\gamma'$ is obtained from $\gamma$ by shifting once at each vertex inside a set $A \subseteq X$. We set $B \coloneqq X-A$ so that $(A,B)$ partitions $X$. We may assume that $|B| \leq |A|$; there is symmetry between $A$ and $B$ since we can also shift at all vertices in $A \cup B$ without changing equation~(\ref{eqn:cor}).

Suppose first that $X = V(G)$. Then since $|V(G)|$ is odd we actually have $|B| \leq |A|-1$. By counting the edges between $A$ and $B$ in two ways, we find that\begin{align*}
    \sum_{v \in A}\deg(v) - 2|E(A)| = \sum_{v \in B}\deg(v) - 2|E(B)| \leq 2\ell|B| \leq 2\ell(|A|-1).
\end{align*}\noindent Since $2\ell(|A|-1)+2d \leq \sum_{v \in A}\deg(v)$, we obtain that $d \leq |E(A)|$. So the flooding number is at least $d$ since $\gamma'(E(X)) \geq |E(A)|$. Thus we may assume that $G-X$ has at least one component. 

Let $Y$ be the vertex-set of a component of $G-X$. Notice that $Y$ ``contributes'' either $|\delta(Y)|/2-1$ or $|\delta(Y)|/2$ to equation~(\ref{eqn:cor}), depending on whether it is $\gamma'$-odd. Since $G$ is $2d$-edge-connected, each component must therefore ``contribute'' at least $d-1$ to equation~(\ref{eqn:cor}). So, since $d \geq 2$, we may assume that $Y = V(G)-X$, that $Y$ is $\gamma'$-odd, and that $|\delta(Y)|=2d$. We may also assume that that $\gamma'(E(X))=0$ (or, equivalently, that $E(A)$ and $E(B)$ are empty). We now aim for a contradiction.

Let $k_1$ and $k_2$ denote the number of edges between $Y$ and $A$ and between $Y$ and $B$, respectively. Using a counting argument for the equality step, we get that\begin{align}
\label{eqn:tightCor}
2\ell(|A|-1)\leq\sum_{v \in A}\deg(v)-2d \leq \sum_{v \in A}\deg(v) - k_1 = \sum_{v \in B}\deg(v) - k_2 \leq 2\ell|B|.
\end{align}\noindent So $|A|-1 \leq |B|$, and if $|B|=|A|-1$ then every step of Inequality~(\ref{eqn:tightCor}) is tight. Recall that $|B|\leq |A|$ by assumption, so in fact there are only two options for the size of $B$; either $|B|=|A|-1$ or $|B|=|A|$. We will split into cases based on the size of $B$, but first we prove the following claim which will be used in both cases.

\begin{claim}
\label{claim:parity}
We have $k_1 +1\equiv \ell|Y|$.
\end{claim}
\begin{proof}
First, notice that $2\ell|Y|=\sum_{y \in Y}\deg(y) =2|E(Y)|+|\delta(Y)|=2|E(Y)|+2d$. So $\ell|Y|=|E(Y)|+d$, and thus $|E(Y)|$ has the same parity as $\ell|Y|+d$. Thus, since $Y$ is $\gamma'$-odd and $|\delta(Y)|=2d$,\begin{align*}
d+1 \equiv \gamma'(E(Y) \cup \delta(Y))\equiv |E(Y)\cup\delta(Y)|+k_1 \equiv |E(Y)|+k_1 \equiv \ell|Y|+d+k_1.
\end{align*}
This proves the claim.
\end{proof}

Now we split into two cases based on the size of $B$. We will obtain a contradiction in both cases.

\smallskip
\noindent\emph{Case 1:} $|B|=|A|-1$
\smallskip

Then $1 \equiv |V(G)|\equiv |A|+|B|+|Y|\equiv 1+|Y|$, and so $|Y|\equiv 0$. Then by Claim~\ref{claim:parity}, $k_1+1\equiv\ell|Y| \equiv 0$. However, since every step of Inequality~(\ref{eqn:tightCor}) is tight, we also have that $\sum_{v \in A}\deg(v)-2d = \sum_{v \in A}\deg(v) - k_1$. So $k_1=2d$. This is a contradiction since we just proved that $k_1+1\equiv 0$.

\smallskip
\noindent\emph{Case 2:} $|B|=|A|$
\smallskip

We may assume that $b \in B$ since now $A$ and $B$ are symmetric. So we have that $k_1-k_2 = 2\ell-2d$; intuitively this is because the extra edges from $A$ (rather than $B$) to $Y$ must make up for the smaller degree of $b$. Formally, this is because\begin{align*}
2\ell|A|-k_1=\sum_{v \in A}\deg(v)-k_1 = \sum_{v \in B}\deg(v)-k_2 = 2\ell(|B|-1)+2d-k_2.
\end{align*}Since $|A|=|B|$, we obtain that $k_1-k_2 = 2\ell-2d$ by simplifying this equation.

We also have the equation $k_1+k_2=|\delta(Y)|=2d$. Summing the corresponding sides, we obtain $2k_1 = 2\ell$. Since $k_1 +1\equiv \ell|Y|$ by Claim~\ref{claim:parity}, we obtain $\ell|Y| \equiv \ell+1$. Thus $\ell$ is odd and $|Y|$ is even. However, we also have that $1 \equiv |V(G)| \equiv |A|+|B|+|Y| \equiv |Y|$ since $|A|=|B|$. This contradicts completes this case and therefore the proof of Corollary~\ref{cor:almostReg}. 
\end{proof}

Now we use Corollary~\ref{cor:almostReg} to prove the other conjecture of M\'{a}\v{c}ajov\'{a} and \v{S}koviera. This corollary was mentioned in the introduction and is re-stated below.

\corMS*
\begin{proof}
Going for a contradiction, suppose not. Choose a counterexample $G$ with as few vertices as possible. Then $G$ is not $2\ell$-edge-connected, since otherwise we could apply Corollary~\ref{cor:almostReg} with any ``root'' vertex. 

Now let $S$ be a set of vertices so that both $S$ and $V(G)-S$ are non-empty; subject to that, choose $S$ so that $|\delta(S)|$ is as small as possible. Since $G$ has an odd number of vertices, we may assume that $|S|$ is odd and $|V(G)-S|$ is even. Since $G$ is Eulerian, there is a positive integer $d$ so that $|\delta(S)|=2d$. Finally, since $G$ is not $2\ell$-edge-connected, we have $d<\ell$.

Now let $H$ be the graph which is obtained from $G$ by identifying $S$ to a single vertex $b$ and then deleting all loops at $b$. (That is, $H$ has vertex-set $(V(G)-S)\cup\{b\}$, the induced subgraph of $H$ on $V(G)-S$ is the same as the induced subgraph of $G$ on $V(G)-S$, and $H$ has one edge with ends $b$ and $x$ for each edge in $\delta(S)$ whose end outside of $S$ is $x$.) We claim that $H$ and $b$ satisfy all of the conditions of Corollary~\ref{cor:almostReg}. The key point is that $H$ is $2d$-edge-connected by the minimality of $|\delta(S)|$ in $G$; for each subset $X$ of $V(H)$ which contains $b$, the number of edges in $\delta(X)$ in $H$ is equal to the number of edges in $\delta(X-\{b\}\cup S)$ in $G$.

Thus, by Corollary~\ref{cor:almostReg}, the graph $H$ has a circuit-decomposition of size~$d$ where each circuit has an odd number of edges and begins and ends at $b$. This yields a collection $\mathcal{T}$ of $d$ trails in $G$ so that\begin{enumerate}
    \item $E(G)-E(S)$ is the disjoint union of the edge-sets of the trails in~$\mathcal{T}$, and
    \item each trail in $\mathcal{T}$ has an odd number of edges and begins and ends in $X$.
\end{enumerate}\noindent Let $G'$ be the graph which is obtained from the subgraph of $G$ induced by $S$ by adding, for each trail $T \in \mathcal{T}$, an edge with the same ends as $T$. 

This graph $G'$ is $2\ell$-regular, has an odd number of vertices, and has fewer vertices than $G$. It is also connected; otherwise, if $Y$ was the vertex-set of one of its components, then in $G$ we would have $|\delta(Y)|<|\delta(S)|$ and a contradiction to the choice of $S$. So, since $G$ is a minimum counterexample to Corollary~\ref{cor:corMS}, the graph $G'$ has a circuit-decomposition $\C'$ of size $\ell$ where all circuits have an odd number of edges and begin and end at the same vertex. We can obtain a circuit-decomposition $\C$ of $G$ with the same properties by replacing the $d$ new edges of $G'$ with the corresponding trails in $\mathcal{T}$. This contradicts the fact that $G$ is a counterexample and completes the proof of Corollary~\ref{cor:corMS}.
\end{proof}

\subsection*{Packing and the Erd\H{o}s-P\'{o}sa property}
Now we prove the corollary from the introduction that relates ``packing'' and ``decomposing''. It is re-stated below for convenience.

\corPacking*
\begin{proof}
Let $(G,\gamma, b)$ be a $4$-edge-connected \rooted{} so that there exists a collection of $\ell$ edge-disjoint non-zero circuits which hit $b$. By Theorem~\ref{thm:mainFlooding}, there is a certificate $(X, \gamma')$ for the flooding number. Moreover, since $G$ is $4$-edge-connected, $\mathrm{odd}_{\gamma'}(G-X) \leq \frac{1}{4}|\delta(X)|$. So \begin{align*}
    2\tilde{\nu}(G, \gamma, b) &=2\gamma'(E(X))+|\delta(X)|-2\mathrm{odd}_{\gamma'}(G-X) \\
    & \geq 2\gamma'(E(X))+\frac{1}{2}|\delta(X)|\\
    & \geq \gamma'(E(X))+\frac{1}{2}|\delta(X)|\\
    & \geq \ell,
\end{align*} since each of the $\ell$ edge-disjoint non-zero circuits which hit $b$ must use either a non-zero edge in $E(X)$, or at least two edges in $\delta(X)$. It follows that $\tilde{\nu}(G, \gamma, b) \geq \lceil \ell/2 \rceil$ since $\tilde{\nu}(G, \gamma, b)$ is an integer. 
\end{proof}

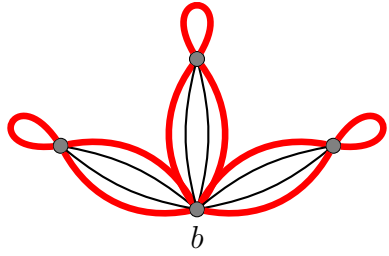
\begin{figure}
\centering
\begin{tikzpicture}[scale=1, every node/.style={MyNode}]
    \def \w {2}
    \node (B1) at (155:\w) {};
    \draw[MyRedArc] (B1) to [out=125,in=185,looseness=25] (B1);
    \node (B2) at (90:\w) {};
    \draw[MyRedArc] (B2) to [out=60,in=120,looseness=25] (B2);
    \node (B3) at (25:\w) {};
    \draw[MyRedArc] (B3) to [out=-5,in=55,looseness=25] (B3);
    \node[label=below:$b$] (b) at (0,0) {};
    \foreach \i in {1,2,3}{
        \draw[MyRedArc] (b) to [bend left=36] (B\i);
        \draw[thick] (b) to [bend left=14] (B\i);
        \draw[thick] (b) to [bend right=14] (B\i);
        \draw[MyRedArc] (b) to [bend right=36] (B\i);
    }
\end{tikzpicture}
\caption{A $4$-edge-connected \rooted{} where the bound in Corollary~\ref{cor:packingCovering} is tight.}
\label{fig:packingTight}
\end{figure}

This proof of Corollary~\ref{cor:packingCovering} also shows how to construct an example where the bound is tight; see Figure~\ref{fig:packingTight}. A similar construction, but where each component of $G-b$ has two edges to $b$, shows that $4$-edge-connectivity is necessary.

We mentioned in the introduction that packing problems have been particularly well-studied in relation to the Erd\H{o}s-P\'{o}sa property. In fact, Corollary~\ref{cor:packingCovering} can be combined with a theorem of Kakimura, Kawarabayashi, and Kobayashi~\cite{packingRooted} to obtain the Erd\H{o}s-P\'{o}sa property for the flooding number of a $4$-edge-connected \rooted{}. The following final corollary of Theorem~\ref{thm:mainFlooding} obtains this type of property directly; Figure~\ref{fig:packingTight} shows that the bounds are tight.

\begin{corollary}
\label{cor:EPP}
If $(G, \gamma, b)$ is a $4$-edge-connected \rooted{} with $\tilde{\nu}(G, \gamma, b)\leq \ell$, then there exists a set $F$ of at most $3\ell$ edges so that $G-F$ has no non-zero circuit which begins and ends at $b$.
\end{corollary}
\begin{proof}
By Theorem~\ref{thm:mainFlooding}, the \rooted{} $(G,\gamma, b)$ has a certificate $(X, \gamma')$. Let $F_1$ be the set of all edges in $E(X)$ which are non-zero according to $\gamma'$, and let $F_2$ be any subset of $\delta(X)$ which is obtained by deleting one edge incident to each component of $G-X$. Then \begin{align*}
    3\ell \geq 3\tilde{\nu}(G, \gamma, b) \geq \gamma'(E(X))+3\left(\frac{|\delta(X)|}{2}-\mathrm{odd}_{\gamma'}(G-X)\right).
\end{align*} 

It is clear that $\gamma'(E(X))=|F_1|$. We claim that $3(|\delta(X)|/2-\mathrm{odd}_{\gamma'}(G-X))\geq |F_2|$. To see this, observe that a component of $G-X$ with vertex-set $Y$ ``contributes'' either $3|\delta(Y)|/2-3$ or $3|\delta(Y)|/2$ to the expression, depending on whether $Y$ is $\gamma'$-odd. Moreover, $3|\delta(Y)|/2-3 \geq |\delta(Y)|-1$ since $G$ is $4$-edge-connected. The corollary follows.
\end{proof}

%%% AUTHOR: optional appendix here
%\appendix %% you may comment this out if no Appendix
%\section*{Appendix}

%%% AUTHOR: optional acknowledgments here
\section*{Acknowledgments} %%  you may comment this out if no Ackno
The author would like to express their deepest thanks to Jim Geelen and Paul Wollan for their input on this paper. The author would also like to thank Louis Esperet for suggesting the connection to the conjectures of M\'{a}\v{c}ajov\'{a} and \v{S}koviera, and to James Davies for feedback which improved the presentation. Finally, the author would like to thank Donggyu Kim for finding a critical mistake in Lemma~\ref{lem:main} in an earlier version of this paper.

%%% AUTHOR:
%%% Bibliography goes here. Note that the arXiv cannot process bibtex
%%% or biber bibliographies.  Example of acceptable bibliograpy format:
\bibliographystyle{amsplain}

%% AUTHOR: You can generate such a bibliography from a .bib file by 
%% running pdflatex/bibtex/pdflatex/pdflatex and then pasting the .bbl file
%% between \begin{thebibliography} and \end{bibliography}

%%% AUTHOR: Include a short description of each author following the
%%% structure below. Use the same short tags used previously.  
%%% Use \imageat{} and \imagedot{} instead of "@" and "." in
%%% email addresses-this replaces the symbols with graphics to avoid 
%%% e-mail address harvesting from the .pdf file
\begin{aicauthors}
\begin{authorinfo}[rm]
  Rose McCarty\\
  School of Mathematics and School of Computer Science\\
  Georgia Institute of Technology\\
  Atlanta, Georgia, USA\\
  rmccarty3\imageat{}gatech\imagedot{}edu \\
  \url{https://mccarty.math.gatech.edu/}
\end{authorinfo}
\end{aicauthors}

\end{document}